\documentclass[aap,preprint]{imsart}

\RequirePackage{amssymb,amsfonts,amsthm,amsmath}
\RequirePackage{bm,bbm}
\RequirePackage{epsfig,graphicx}
\RequirePackage[colorlinks,citecolor=blue,urlcolor=blue]{hyperref}
\RequirePackage{comment}
\arxiv{arXiv:0000.0000}
\usepackage{lscape}
\usepackage[ruled,vlined]{algorithm2e}
\usepackage{mathtools}
\usepackage{float}

\startlocaldefs
\theoremstyle{plain}
\newtheorem{theorem}{Theorem}
\newtheorem{lemma}{Lemma}
\newtheorem{assumption}{Assumption}

\def \sd  {{\underline{s}}}
\def \td  {{\underline{t}}}
\def \ud  {{\underline{u}}}
\def \Td  {{\underline{T}}}

\def \hX  {{\widehat{X}}}

\def \bX  {{\overline{X}}}

\def \hM  {{\widehat{M}}}
\def \bM  {{\overline{M}}}

\def \tn  {{t_n}}
\def \tnp {{t_{n+1}}}

\def \EE  {{\mathbb{E}}}
\def \PP  {{\mathbb{P}}}
\def \RR  {{\mathbb{R}}}

\def \D   {{\rm d}}
\def \DW  {{\Delta W}}

\def \hmax {{h_{\rm max}}}

\def \halfs {{\textstyle \frac{1}{2}}}

\def \e   {{\rm e}}  

\newcommand{\fracs}[2]{{\textstyle \frac{#1}{#2}}}

\endlocaldefs

\begin{document}

\begin{frontmatter}
\title{Adaptive Euler-Maruyama method for\\ 
       SDEs with non-globally Lipschitz drift:\\
       part II, infinite time interval}
\runtitle{Adaptive Euler-Maruyama method for non-Lipschitz drifts}

\begin{aug}
\author{\fnms{Wei} \snm{Fang}\ead[label=e1]{wei.fang@maths.ox.ac.uk}}
\and
\author{\fnms{Michael~B.} \snm{Giles}\ead[label=e2]{mike.giles@maths.ox.ac.uk}}

\runauthor{W.~Fang \& M.B.~Giles}

\affiliation{University of Oxford}
\address{Mathematical Institute\\University of Oxford\\Oxford OX2 6GG\\United Kingdom\\
\printead{e1}\\
\phantom{E-mail:\ }\printead*{e2}}
\end{aug}

\begin{abstract}
This paper proposes an adaptive timestep construction 
for an Euler-Maruyama approximation of the ergodic SDEs with a drift 
which is not globally Lipschitz over an infinite time interval. If the timestep is bounded appropriately, we show not only the stability of the numerical solution and the standard strong convergence order, but also that the bound for moments and strong error of the numerical solution are uniform in $T,$ which allow us to introduce the adaptive multilevel Monte Carlo. Numerical experiments support our analysis.    
\end{abstract}



\end{frontmatter}

\section{Introduction}

In this paper we consider an $m$-dimensional stochastic 
differential equation (SDE) driven by a $d$-dimensional 
Brownian motion:
\begin{equation}
\D X_t = f(X_t)\,\D t + g(X_t)\,\D W_t,
\label{SDE}
\end{equation}
which has a non-globally Lipschitz drift $f: \RR^m\!\rightarrow\!\RR^m$ satisfying the dissipativity condition: for some $\alpha,\beta>0,$
\begin{equation}
\langle x,f(x)\rangle \leq -\alpha \|x\|^2+\beta,
\label{Intro Contractive}  
\end{equation}
and a bounded volatility 
$g: \RR^m\!\rightarrow\!\RR^{m\times d}.$ In particular, we focus on a class of SDEs which are ergodic in nature and converge exponentially to some invariant measure $\pi$. Evaluating the expectation of some function $\varphi(x)$ with respect to that invariant measure $\pi$ is of great interest in mathematical biology, physics and Bayesian inference in statistics:
$$\pi(\varphi)\coloneqq\int\varphi(x)\,\mathrm{d}\pi(x)=\lim_{t\rightarrow \infty}\mathbb{E}\left[\varphi(X_t) \right],\ \ \varphi\in L^1(\pi). $$
Several different methodologies have been developed to estimate it.\par
First, we can compute the probability density function $\rho(x)$ of $\pi$ by solving the corresponding stationary Fokker-Planck equation, see \cite{soize1994fokker}. However, the stationary Fokker-Planck equation is a partial differential equation (PDE) and its numerical solution becomes extremely expensive when the dimension of the PDEs becomes large.\par
The second approach is based on the ergodicity of the SDEs:
\begin{equation}
\lim_{T\rightarrow\infty} \frac{1}{T}\int_0^T \varphi(X_t)\,\mathrm{d}t\,=\, \pi(\varphi),\ a.s.,
\label{timeaverage}
\end{equation} 
where the limit does not depend on initial value $x_0.$ This approach uses discretized numerical schemes to approximate the SDEs and requires the numerical solution $\widehat{X}_t$ to preserve the ergodicity. In practice, we can choose a sufficiently large $N$ and compute
$$\frac{1}{N}\sum_{n=1}^N\varphi(\hX_{nh}),$$
where $\hX_{nh}$ is the numerical solution at the $n$th discretized time point using an ergodic method with a uniform timestep $h.$ Under the dissipativity condition (\ref{Intro Contractive}) together with the Lipschitz condition for $f$, Talay \cite{talay1990second} shows the standard weak convergence order for the Milstein method:
$$\lim_{N\rightarrow\infty}\frac{1}{N}\sum_{n=1}^N\varphi(\hX_{nh})=\int \varphi(x)\,\mathrm{d}\pi(x)+O(h).$$
Roberts $\&$ Tweedie in \cite{roberts1996exponential} analyse the ergodicity of the unadjusted Langevin algorithm for the Langevin equation,  which has uniform volatility and satisfies dissipativity condition (\ref{Intro Contractive}). This scheme corresponds to the standard Euler-Maruyama method:
\[\hX_{(n+1)h}=\hX_{nh}+f(\hX_{nh})\,h+\Delta W_n\]
using a uniform timestep of size $h$ with Brownian increments $\Delta W_n.$ The paper shows that the numerical solution is not ergodic when $f$ has a polynomial degree larger than 1. Metropolis-adjusted Langevin algorithm (MALA) is introduced but the numerical solutions are still not exponentially ergodic for non-linear drift $f.$ 

For globally Lipschitz Langevin SDEs satisfying the dissipativity condition (\ref{Intro Contractive}), the standard Euler-Maruyama method is shown in \cite{MSH02} to inherit ergodicity provided the timesteps are sufficiently small.
However, the standard Euler-Maruyama method and Milstein method fail to be stable for non-globally Lipschitz SDEs. The Split-step backward Euler method 
\begin{eqnarray*}
\hX_{nh}^*     &=& \hX_{nh} + f(\hX_{nh}^*)\, h,\\
\hX_{(n\!+\!1)h} &=& \hX_{nh}^* + g(\hX_{nh}^*)\, \Delta W_n.
\end{eqnarray*}
and the drift-implicit Backward Euler method:
\[
\hX_{(n\!+\!1)h} = \hX_{nh} + f(\hX_{(n\!+\!1)h})\, h + g(\hX_{nh})\, \Delta W_n.
\]
are proved to be ergodic in \cite{MSH02}.\par

Under the same conditions, Hansen in \cite{hansen2003geometric} considers the local linearization of the drift coefficient, that is the first-order Taylor approximation of $\tilde{X}_t$: for $t\in [nh,(n+1)h],$ given $\tilde{X}_{nh}=x,$
\begin{equation*}
\D\tilde{X}_t=(f(x)+\nabla f(x)(\tilde{X}_t-x))\,\D t+\D W_t,
\label{locallinear}
\end{equation*}
which is an Ornstein-Uhlenbeck diffusion and we can calculate the analytical conditional distribution of $\tilde{X}_{(n+1)h}.$ However, this analytical treatment only applies when the diffusion coefficient is uniform.

An adaptive timestepping algorithm proposed by Lamba, Mattingly $\&$ Stuart in \cite{lamba2006adaptive} chooses the step size by halving or doubling based on the local error estimation and a user-input tolerance $\tau$. More precisely,
\begin{equation*}
\begin{aligned}
\widehat{X}_{t_n}^*&= \widehat{X}_{t_n}+ f(\widehat{X}_{t_n})\,h_n \\
\widehat{X}_{t_{n+1}}&=\widehat{X}_{t_n}^*+g(\widehat{X}_{t_n})\,\Delta W_n
\end{aligned}
\end{equation*}
where $h_n=2^{-k_n}\hmax$ satisfies that
$$h_n\leq \min(2h_{n-1},\hmax),\ \ k_n=\min(k\in\mathbb{Z}:\ |f(\widehat{X}_{t_n}^*)-f(\widehat{X}_{t_n})|\leq \tau).$$
This scheme is proved to preserve the ergodicity of original SDEs under the dissipativity condition (\ref{Intro Contractive}) and boundedness and invertibility of the diffusion coefficient $g.$

Lemaire \cite{lemaire2007}
considers an infinite time interval under the dissipativity condition generated by a general Lyapunov function using a timestep with an 
upper bound which decreases towards zero over time, and proves 
convergence of the empirical distribution to the invariant 
distribution of the SDE.

Finally, without requiring the ergodicity of the schemes, for exponentially ergodic SDEs, we can choose a sufficiently large $T$ such that $$|\mathbb{E}[\varphi(X_T)]-\pi(\varphi)|\leq \varepsilon.$$
Then, for this fixed $T,$ we can use of all the methods mentioned in Part I of this pair of articles \cite{Part1} to estimate $\mathbb{E}[\varphi(X_T)].$ Milstein $\&$ Tretyakov \cite{milstein2007computing} analyse the error of this kind of approach based on their quasi-symplectic method. In practice, a suitable choice of initial data is important because the transition period to a sufficient proximity of the equilibrium can be rather long. Therefore, running a small number of pioneer paths can be employed to obtain a good initial distribution for the overwhelming majority of simulations.

We should remark here that the PDE approach is far too expensive in high dimensions. The time-averaging approach (\ref{timeaverage}) requires the numerical methods to preserve the ergodicity but the third approach does not. The length of the time interval $[0,T]$ used in the time-averaging approach is much longer than the third approach, for it needs not only to ensure that the distribution of $\hX_T$ is sufficiently close to the invariant measure $\pi,$ but also to guarantee a small variance for the average. Therefore, the second approach needs to simulate a single long path but the third one needs to simulate a lot of relative short paths, which allows multilevel Monte Carlo (MLMC) and parallel computing techniques to be employed. One important concern about the third approach is that in the numerical analysis of existing algorithms in the finite time interval $[0,T],$ the strong error increases exponentially as $T$ increases. This is not acceptable when we need to simulate a much larger $T$ and it will be a key concern in this paper. The final issue about the second and third approaches is how to choose a good $T$ such that the weak error is bounded appropriately.    
 
In this paper, we propose instead to use the standard explicit 
Euler-Maruyama method, but with an adaptive timestep $h_n$ which 
is a function of the current approximate solution $\hX_{t_n}$. By setting a suitable condition for $h,$ we can show that, instead of an exponential bound, the numerical solution has a uniform bound with respect to $T$ for both moments and the strong error. Then, MLMC methodology \cite{giles08,giles15} is employed and non-nested timestepping is used to construct an adaptive MLMC \cite{GLW16}. Following the idea of Glynn and Rhee \cite{glynn2014exact} to estimate the invariant measure of some Markov chains, we introduce an adaptive MLMC algorithm for the infinite interval, in which each level $\ell$ has a different time interval length $T_\ell,$ to achieve a better computational performance. Note that using different time interval lengths allows us not to worry how to choose an appropriate $T$ before simulation. The MLMC algorithm will automatically terminate at a level $L$ with a sufficiently large $T_L.$ 

The rest of the paper is organised as follows.  Section 2 states
the main theorems and proves some minor lemmas.  Section 3 introduces the MLMC schemes, and the relevant numerical experiments are provided in section 4. The proofs of the three main theorems are deferred to section 5, and finally, section 6 has some 
conclusions and discusses future extensions.

In this paper we consider the infinite time interval $[0,\infty)$ and let 
$(\Omega,\mathcal{F},\PP)$ be a probability space with 
normal filtration $(\mathcal{F}_t)_{t\in[0,\infty)}$ for section 2 and $(\mathcal{F}_t)_{t\in(-\infty,0]}$ for section 3 corresponding to
a $d$-dimensional standard Brownian motion 
$W_t=(W^{(1)},W^{(2)},\ldots,W^{(d)})_t.$
We denote the vector norm by 
$\|v\|\triangleq(|v_1|^2+|v_2|^2+\ldots+|v_m|^2)^{\frac{1}{2}}$, 
the inner product of vectors $v$ and $w$ by 
$\langle v,w \rangle\triangleq v_1w_1+v_2w_2+\ldots+v_mw_m$, 
for any $v,w\in\RR^m$ and the Frobenius matrix norm by 
$\|A\|\triangleq \sqrt{\sum_{i,j}A_{i,j}^2}$ 
for any $A\in\RR^{m\times d}.$

\section{Adaptive algorithm and theoretical results}

\subsection{Adaptive Euler-Maruyama method}

The adaptive Euler-Maruyama discretisation is
\begin{equation}
\tnp = \tn + h_n, ~~~~
\hX_\tnp = \hX_\tn + f(\hX_\tn)\, h_n + g(\hX_\tn)\, \DW_n,
\label{eq:dis_SDE}
\end{equation}
where 
$h_n\triangleq h(\hX_\tn)$ and $\DW_n \triangleq W_\tnp \!-\! W_\tn$,
and there is fixed initial data $t_0\!=\!0,\ \hX_0\!=\!x_0$.

We use the notation 
$
\td \triangleq \max\{t_n: t_n\!\leq\! t\},\
n_t\triangleq\max\{n: t_n\!\leq\! t\}
$
for the nearest time point before time $t$, and its index.

We define the piecewise constant interpolant process $\bX_t=\hX_\td$ 
and also define the standard continuous interpolant \cite{kp92} as
\begin{equation}
\hX_t=\hX_\td+f(\hX_\td) (t\!-\!\td) + g(\hX_\td) (W_t\!-\!W_\td),
\label{eq:con_SDE}
\end{equation}
so that $\hX_t$ is the solution of the SDE
\begin{equation}
\D \hX_t =  f(\hX_\td)\, \D t + g(\hX_\td) \, \D W_t= f(\bX_t)\,\D t+ g(\bX_t)\,\D W_t.
\label{eq:approx_SDE}
\end{equation}

In the following subsections, we state some preliminary lemmas, the key results on 
stability and strong convergence, and related results on 
the number of timesteps, introducing various assumptions 
as required for each.  The main proofs are deferred to 
Section \ref{sec:proofs}.

\subsection{Stability}

\begin{assumption}[Local Lipschitz and linear growth]
\label{assp:linear_growth}
$f$ and $g$ are both locally Lipschitz, so that for any 
$R\!>\!0$ there is a constant $C_R$ such that
\[
\| f(x)\!-\!f(y) \| + \| g(x)\!-\!g(y) \| \leq C_R\, \|x\!-\!y\| 
\]
for all $x,y\in \RR^m$ with $\|x\|,\|y\|\leq R$.
Furthermore, there exist constants $\alpha, \beta > 0$ 
such that for all $x\in \RR^m$, $f$ satisfies the dissipativity condition:
\begin{equation}
\langle x,f(x)\rangle \leq -\alpha\|x\|^2+\beta,
\label{eq:onesided_growth}
\end{equation}
and $g$ is globally bounded and non-degenerate:
\begin{equation}
\|g(x)\|^2\leq \beta.
\label{eq:g_bound}
\end{equation} 
\end{assumption}

\begin{lemma}[SDE stability]

\label{lemma:SDE_stability}

If the SDE satisfies Assumption \ref{assp:linear_growth} with $X_0=x_0,$
then for all $p\in(0,\infty),$ there is a constant $C_p$ which only depends on $x_0$ and $p$ such that, $\forall\,t\geq 0,$  
\[
\EE\left[ \|X_t\|^p \right] \leq C_p.
\]
\end{lemma}

\begin{proof}
The proof is for 
$p\!\geq\!2$; the result for $0\!<\! p \!<\! 2$ follows 
from H{\"o}lder's inequality.

By It\^{o}'s formula, we have for any $t\geq 0,$
\begin{eqnarray*}
\e^{p\alpha t/2}\|X_t\|^p-\|X_0\|^p&\leq& \int_0^t p\left(\frac{\alpha}{2}\|X_t\|^2+\langle X_s, f(X_s)\rangle\right) \e^{p\alpha s/2}\|X_s\|^{p-2} \D s\\
&& + \int_0^t \frac{p(p-1)}{2}\|g(X_s)\|^2\e^{p\alpha s/2}\|X_s\|^{p-2} \D s\\
&& +\int_0^t p\,\e^{p\alpha s/2}\|X_s\|^{p-2}\langle X_s,g(X_s)\D W_s\rangle,
\end{eqnarray*}
and then by taking expectations on both sides and using the dissipativity property (\ref{eq:onesided_growth}) and boundedness (\ref{eq:g_bound}), we obtain
\begin{eqnarray*}
\EE\left[\e^{p\alpha t/2}\|X_t\|^p\right]-\EE\left[\|X_0\|^p\right] \!\!\!
&\leq& \!\!\!  \int_0^t -\frac{p\alpha}{2}\, \EE\left[\e^{p\alpha s/2}\|X_s\|^{p}\right] \D s\\
&&+\int_0^t \EE\left[\frac{p(p+1)\beta}{2} \e^{p\alpha s/2}\|X_s\|^{p-2}\right] \D s.
\end{eqnarray*}
Young inequality (\ref{Young inequality}) implies that
\[
\frac{p(p+1)\beta}{2} \,\e^{p\alpha s/2}\,\|X_s\|^{p-2}\leq \frac{p\alpha}{2}\,\e^{p\alpha s/2}\|X_s\|^p + c_p\,\e^{p\alpha s/2}
\]
where $c_p=\left(\frac{p-2}{p\alpha}\right)^{p/2-1}(\beta(p+1))^{p/2}.$
Therefore, we obtain
\begin{eqnarray*}
\EE\left[\e^{p\alpha t/2}\|X_t\|^p\right]-\EE\left[\|X_0\|^p\right] \!\!\! 
&\leq &\!\!\! \int_0^t c_p\, \e^{p\alpha s/2}\, \D s,
\end{eqnarray*}
and we can conclude that
\[
\EE\left[\|X_t\|^p\right]\leq \frac{2c_p}{p\alpha}+\e^{-p\alpha t/2}\EE\left[\|X_0\|^p\right]\leq \frac{2c_p}{p\alpha}+ \|x_0\|^p \triangleq C_p.
\]
\end{proof}

\vspace{-0.1in}

We now specify the critical assumption about the adaptive timestep.

\begin{assumption}[Adaptive timestep] 
\label{assp:timestep}

The adaptive timestep function $h: \RR^m\rightarrow (0,\hmax]$ is 
continuous and bounded, with $0<\hmax<\infty$, and there exist constants $\alpha, \beta > 0$ 
such that for all 
$x\in \RR^m$, $h$ satisfies the inequality
\begin{equation}
\langle x, f(x) \rangle + \halfs \, h(x)\, \|f(x)\|^2 \leq -\alpha \|x\|^2 + \beta.
\label{eq:timestep}
\end{equation}
\end{assumption}
Note that if another timestep function $h^\delta(x)$ is smaller than 
$h(x)$, then  $h^\delta(x)$ also satisfies this Assumption.  Note 
also that the form of (\ref{eq:timestep}), which is motivated by 
the requirements of the proof of the next theorem, is very similar 
to (\ref{eq:onesided_growth}).  Indeed, if (\ref{eq:timestep}) is 
satisfied then (\ref{eq:onesided_growth}) is also true for the 
same values of $\alpha$ and $\beta$. Compared with the condition in the finite time analysis \cite{Part1}, we need the additional $\hmax$ upper bound to achieve the following result.

\begin{theorem}[Infinite time stability]

\label{thm:stability}

If the SDE satisfies Assumption \ref{assp:linear_growth}, and the 
timestep function $h$ satisfies Assumption \ref{assp:timestep},
then for all $p\in(0,\infty)$ there exists a constant $C_{p}$ which 
depends solely on $p,$ $x_0$, $\hmax$ and the constants $\alpha, \beta$ in
Assumption \ref{assp:timestep} such that, $\forall t\geq 0,$
\[
\EE\left[ \|\hX_t\|^p \right] < C_{p},\ \EE\left[ \|\bX_t\|^p \right] < C_{p}.
\]
\end{theorem}

\begin{proof}
The proof is deferred to Section \ref{sec:proofs}.
\end{proof}

To bound the expected number of timesteps, we require an assumption 
on how quickly $h(x)$ can approach zero as $\|x\|\rightarrow\infty$.

\begin{assumption}[Timestep lower bound]
\label{assp:lower_bound}
There exist constants $\xi,\zeta,q\!>\!0$, such that the adaptive timestep 
function satisfies the inequality
\[
h(x) \geq \left( \xi \|x\|^q + \zeta \right)^{-1}.
\]
\end{assumption}

Given this assumption, we obtain the following lemma.

\begin{lemma}[Bounded timestep moments]
\label{lemma: BoundedTimestep}
If the SDE satisfies Assumption \ref{assp:linear_growth}, and the 
timestep function $h$ satisfies Assumptions \ref{assp:timestep}
and \ref{assp:lower_bound}, then for all $T,p\in(0,\infty)$ there exists a constant $C_{h,p}$  which 
depends on $p$ and $C_p$ in Theorem \ref{thm:stability} such that
\[
\EE\left[ (N_T-1)^p \right] < C_{h,p}\,T^p.
\]
where
$\displaystyle
N_T = \min\{n: t_n\geq T\}
$
is the number of timesteps required by a path approximation.
\end{lemma}

\begin{proof}
By Assumption \ref{assp:lower_bound}, we have
\[
N_T=\sum_{k=1}^{n_T}1=\sum_{k=1}^{n_T}\frac{h(\hX_{t_k})}{h(\hX_{t_k})}=\int_0^{\Td}\frac{1}{h(\bX_t)} \D t+1\leq \int_0^{T} (\xi \|\bX_t\|^q + \zeta) \D t+1
\]
Therefore, by Jensen's inequality, we obtain
\[
\EE\left[(N_T-1)^p\right]\leq T^{p-1}\int_0^T\EE\left[\left(\xi \|\bX_t\|^q+\zeta\right)^{p}\right]\D t 
\]
and the result is then an immediate consequence of Theorem \ref{thm:stability}.
\end{proof}

\subsection{Strong convergence}

Standard strong convergence analysis for an approximation 
with a uniform timestep $h$ considers the limit $h\!\rightarrow\!0$.
This clearly needs to be modified when using an adaptive timestep,
and we will instead consider a timestep function $h^\delta$ controlled
by a scalar parameter $0\!<\!\delta\!\leq\!1$, and consider the
limit $\delta\!\rightarrow\!0$.

In our analysis, we will make the 
following assumption.

\begin{assumption}
\label{assp:timestep_delta}

The timestep function $h^\delta$ satisfies the inequalities
\begin{equation}
\label{eq:h_delta}
\delta\, \min(\hmax, h(x)) \leq h^\delta(x) \leq \min( \delta\,\hmax, h(x) ), 
\end{equation}
and $h$ satisfies Assumption \ref{assp:timestep}.
\end{assumption}

To prove an order of strong convergence requires new assumptions 
on $f$ and $g$:

\begin{assumption}[Contractive Lipschitz properties]

\label{assp:ContractionLipschitz}

For some fixed $p^*\in[2,\infty),$ there exist constants $\lambda,\eta\!>\!0$ such that for all $x,y\in\RR^m$,
$f$ and $g$ satisfy the contractive Lipschitz condition:
\begin{equation}
\langle x\!-\!y,f(x)\!-\!f(y)\rangle\!+\frac{p^*-1}{2}\,\|g(x)\!-\!g(y)\|^2\! \leq -\lambda\,\|x\!-\!y\|^2,
\label{eq:contraction_Lipschitz}
\end{equation}
and $g$ satisfies the Lipschitz condition:
\begin{equation}
\|g(x)\!-\!g(y)\|^2\! \leq \eta\,\|x\!-\!y\|^2,
\label{eq:g_Lipschitz}
\end{equation}
In addition, $f$ satisfies the local polynomial growth Lipschitz condition
\begin{equation}
\|f(x)\!-\!f(y)\| \leq \left(\gamma\, (\|x\|^q \!+\! \|y\|^q) + \mu\right) \, \|x\!-\!y\|,
\label{eq:local_Lipschitz}
\end{equation}
for some $\gamma, \mu, q > 0$.
\end{assumption}

Note that we will prove that this Assumption ensures that two solutions to this SDE starting from different places but driven by the same Brownian motion, will come together exponentially. That means the error made on previous time steps will decay exponentially and then we can prove a uniform bound for the strong error. If the drift and volatility are differentiable, the following assumption is equivalent to Assumption 
\ref{assp:ContractionLipschitz}, and usually easier to check in practice.

\begin{assumption}[Contractive Lipschitz properties]

\label{assp:Lipschitz_diff}
For some fixed $p^*\in[2,\infty),$
there exists a constant $\lambda\!>\!0$ such that for 
all $x,e\in\RR^m$ with $\|e\|\!=\!1$,
$f$ and $g$ are differentiable and satisfy the contractive Lipschitz condition:
\begin{equation}
\langle e ,\nabla f(x)\, e \rangle+\frac{p^*-1}{2}\,\| \nabla g(x)\|^2 \leq -\lambda
\label{eq:contraction_Lipschitz_diff}
\end{equation}
and $g$ satisfies the Lipschitz condition:
\begin{equation}
\|\nabla g(x)\|^2\! \leq \eta,
\label{eq:g_Lipschitz_diff}
\end{equation}
and in addition $f$ satisfies the local polynomial growth 
Lipschitz condition
\begin{equation}
\|\nabla f(x) \| \leq  2\,\gamma\, \|x\|^q + \mu,
\label{eq:local_Lipschitz_diff}
\end{equation}
for some $\gamma, \mu, q > 0$.
\end{assumption}
\begin{lemma}[SDE contractivity]

\label{lemma:SDE_contractivity}

If the SDE satisfies Assumption \ref{assp:ContractionLipschitz},
then for $p\in[2,p^*]$ any two solutions to the SDE: $X_t$ and $Y_t,$ driven by the same Brownian motion but starting from $x_0$ and $y_0$, where $x_0\neq y_0$, satisfy, $\forall\ t>0,$  
\[
\EE\left[ 
 \|X_t-Y_t\|^p \right] \leq \e^{-\lambda p t}\,\EE\left[\|X_0-Y_0\|^p\right] .
\]
\end{lemma}

\begin{proof}
First, we can define $e_t\triangleq X_t-Y_t,$ and since $X_t$ and $Y_t$ are driven by the same Brownian motion, we get
\[\D e_t = \left(f(X_t)-f(Y_t)\right)\D t+\left(g(X_t)-g(Y_t)\right)\D W_t\]
By It\^{o}'s formula, we have for any $0< t\leq T,$
\begin{eqnarray*}
\e^{\lambda p t}\|e_t\|^p-\|e_0\|^p\!\!\!\!\! &\leq& \!\!\!\!\! \int_0^t \lambda p\, \e^{\lambda p s}\|e_s\|^p\,\D s  +\int_0^t p \langle e_s, f(X_s)-f(Y_s) \rangle \e^{\lambda p s}\|e_s\|^{p-2}\,\D s \\
&&+\int_0^t \frac{p(p-1)}{2}\|g(X_s)-g(Y_s)\|^2 \e^{\lambda p s}\|e_s\|^{p-2} \D s\\
&&+ \int_0^t p\,\e^{\lambda p s} \|e_s\|^{p-2}\langle e_s, \left(g(X_s)-g(Y_s)\right)\D W_s \rangle.
\end{eqnarray*}
Therefore, by taking expectations on both sides and using the contractive Lipschitz property (\ref{eq:contraction_Lipschitz}), we obtain that
\[\EE\left[\e^{\lambda p t}\|e_t\|^p\right]-\EE\left[\|e_0\|^p\right]\leq 0.\]
\end{proof}

\begin{theorem}[Strong convergence order]

\label{thm:convergence_order}

If the SDE satisfies Assumption \ref{assp:ContractionLipschitz}, and the timestep 
function $h^\delta$ satisfies Assumption \ref{assp:timestep_delta},
then for all $p\in(0,p^*]$ there exists a constant $C_{p}$ such that, $\forall t\geq 0,$
\[
\EE\left[ \, \|\hX_t\!-\!X_t\|^p \right] \leq C_{p} \, \delta^{p/2}.
\]
\end{theorem}

\begin{proof}
The proof is deferred to Section \ref{sec:proofs}.
\end{proof}
Note that this theorem implies the half-order weak convergence, but the numerical results shows the standard first order weak convergence. We do not provide the proof for the first order weak convergence since the strong convergence result is sufficient to extend this scheme to MLMC. 

\begin{lemma}[Number of timesteps]

\label{lemma:timesteps2}

If the SDE satisfies Assumption \ref{assp:linear_growth}, and the timestep 
function $h^\delta$ satisfies  Assumption \ref{assp:timestep_delta} with $h(x)$ satisfying Assumption \ref{assp:lower_bound},
then for all $T,p\in(0,\infty)$ there exists a constant $C_{h,p}$ same as in Lemma \ref{lemma: BoundedTimestep} such that
\[
\EE\left[ (N_T-1)^p \right] \leq C_{h,p}\,T^p \, \delta^{-p}.
\]
where $N_T$ is again the number of timesteps required by a path approximation.
\end{lemma}

\begin{proof}
The proof is very similar to the proof of Lemma \ref{lemma: BoundedTimestep}, noting that
\[h^\delta (x)\geq \delta\, h(x)\geq \delta\left(\xi\|x\|^q+\zeta\right)^{-1},\]
due to Assumptions \ref{assp:lower_bound} and \ref{assp:timestep_delta}.
\end{proof}

First order strong convergence is achievable for Langevin SDEs
in which $m\!=\!d$ and $g$ is the identity matrix $I_m$, but 
this requires stronger assumptions on the drift $f$.

\begin{assumption}[Enhanced contractive Lipschitz properties]

\label{assp:enhanced_Lipschitz}

There exists a constant $\lambda\!>\!0$ such that for all $x,y\in\RR^m$,
$f$ satisfies the contractive one-sided Lipschitz condition:
\begin{equation}
\langle x\!-\!y,f(x)\!-\!f(y)\rangle \leq -\lambda\|x\!-\!y\|^2.
\label{eq:onesided_Lipschitz2}
\end{equation}
In addition, $f$ is differentiable, and $f$ and $\nabla f$ 
satisfy the local polynomial growth Lipschitz condition
\begin{equation}
\|f(x)\!-\!f(y)\| + \|\nabla f(x)\!-\!\nabla f(y)\|\leq 
\left( \gamma\, (\|x\|^q \!+\! \|y\|^q) + \mu\right) \|x\!-\!y\|,
\label{eq:local_Lipschitz2}
\end{equation}
for some $\gamma, \mu, q > 0$.
\end{assumption}

\begin{lemma}
\label{lemma:useful}
If $f$ satisfies Assumption \ref{assp:enhanced_Lipschitz},
then for any $x,y,v \in \RR^m$\\[-0.1in]
\[
\langle v, f(x)\!-\!f(y) \rangle = 
\langle v, (x\!-\!y)\cdot\nabla f(x) \rangle + R(x,y,v),
\]
where the remainder term has the bound
\[
|R(x,y,v)| \leq \left( \gamma\, (\|x\|^q \!+\! \|y\|^q) + \mu\right) \|v\| \, \|x\!-\!y\|^2.
\]
\end{lemma}
\begin{proof}
If we define the scalar function $u(\lambda)$ for $0\!\leq\!\lambda\!\leq\!1$ by
\[
u(\lambda) = \langle v, f(y+\lambda(x\!-\!y)) \rangle,
\]
then $u(\lambda)$ is continuously differentiable, and by the Mean Value Theorem
$u(1)\!-\!u(0)= u'(\lambda^*)$ for some $0\!<\!\lambda^*\!<\!1$, which implies 
that
\[
\langle v, f(x)\!-\!f(y) \rangle = \langle v, (x\!-\!y)\cdot \nabla f(y+\lambda^*(x\!-\!y)) \rangle.
\]
The final result then follows from the Lipschitz property of $\nabla f$.
\end{proof}

We now state the theorem on improved strong convergence.

\begin{theorem}[Strong convergence for Langevin SDEs]

\label{thm:convergence_order2}

If $m\!=\!d$, $g\equiv I_m$, $f$ satisfies 
Assumption \ref{assp:enhanced_Lipschitz}, 
and the timestep function $h^\delta$ satisfies Assumption 
\ref{assp:timestep_delta}, then for all $p\in(0,\infty)$ 
there exists a constant $C_{p}$ such that, $\forall\ t\geq 0,$
\[
\EE\left[ \|\hX_t\!-\!X_t\|^p \right] \leq C_{p} \, \delta^p.
\]
\end{theorem}

\begin{proof}
The proof is deferred to Section \ref{sec:proofs}.
\end{proof}

\section{Adaptive Multilevel Monte Carlo for invariant distributions}
We are interested in the problem of approximating:
\[\pi(\varphi):=\EE_{\pi}\varphi=\int_{\mathbb{R}^m}\varphi(x)\pi(\D x),\] where $\pi$ is the invariant measure of the SDE (\ref{SDE}). Numerically, we can approximate this quantity by simulating $\EE\left[\varphi(X_T)\right]$ for a sufficiently large $T.$ In the following subsections, we will introduce our adaptive multilevel Monte Carlo algorithm and its numerical analysis.
\subsection{Algorithm}
To estimate $\EE\left[\varphi(X_T)\right],$ the simplest Monte Carlo estimator is 
\[\frac{1}{N}\sum_{n=1}^N \varphi(\hX_T^{(n)}),\]
where $\hX_T^{(n)}$ is the terminal value of the $n$th numerical path in the time interval $[0,T]$ using a suitable adaptive function $h^\delta.$ It can be extended to Multilevel Monte Carlo by using non-nested timesteps as explained in \cite{GLW16}. Consider the identity
\begin{equation}
\EE\left[\varphi_L\right]=\EE\left[\varphi_0\right]+\sum_{\ell=1}^L\EE\left[\varphi_\ell-\varphi_{\ell-1}\right],
\label{MLMC identity}
\end{equation}
where $\varphi_\ell:=\varphi(\hX_T^\ell)$ with $\hX_T^\ell$ being the numerical estimator of $X_T,$ which uses adaptive function $h^{\delta}$ with $\delta\!=\!M^{-\ell}$ for some fixed $M\!>\!1.$ Then the standard MLMC estimator is the following telescoping sum:
\[\frac{1}{N_0}\sum_{n=1}^{N_0} \varphi(\hX_T^{(n,0)})+ \sum_{\ell=1}^L\left\{ \frac{1}{N_\ell}\sum_{n=1}^{N_\ell} \left( \varphi(\hX_T^{(n,\ell)}) -\varphi(\hX_T^{(n,\ell-1)}) \right)\right\},\]
where $\hX_T^{(n,\ell)}$ is the terminal value of the $n$th numerical path in the time interval $[0,T]$ using a suitable adaptive function $h^\delta$ with $\delta=M^{-\ell}.$

Different from the standard MLMC with fixed time interval $[0,T],$ we now allow different levels to have a different length of time interval $T_\ell,$ satisfying
$0<T_0<T_1<...<T_\ell<...<T_L=T,$ which means that as level $\ell$ increases, we obtain a better approximation not only by using smaller timesteps but also by simulating a longer time interval. However, the difficulty is how to construct a good coupling on each level $\ell$ since the fine path and coarse path have different lengths of time interval $T_\ell$ and $T_{\ell-1}.$ 

Following the idea of Glynn and Rhee \cite{glynn2014exact} to estimate the invariant measure of some Markov chains, we perform the coupling by starting a level $\ell$ fine path simulation at time $t^f_0=-T_\ell$ and a coarse path simulation at time $t^c_0=-T_{\ell-1}$ and terminate both paths at $t=0.$ Since the drift $f$ and volatility $g$ do not depend explicitly on time $t,$ the distribution of the numerical solution simulated on the time interval $[-T_\ell,0]$ is same as one simulated on $[0,T_\ell].$ The key point here is that the fine path and coarse path share the same driving Brownian motion during the overlap time interval $[-T_{\ell-1},0].$ Owing to the result of Lemma \ref{lemma:SDE_contractivity}, two solutions to the SDE satisfying Assumption \ref{assp:ContractionLipschitz}, starting from different initial points and driven by the same Brownian motion will converge exponentially. Therefore, the fact that different levels terminate at the same time is crucial to the variance reduction of the multilevel scheme. 

Our new multilevel scheme still has the identity (\ref{MLMC identity}) but with $\varphi_\ell = \varphi(\hX_0^\ell)$ with $\hX^\ell_0$ being the terminal value of the numerical path approximation on the time interval $[-T_\ell,0]$ using adaptive function $h^\delta$ with $\delta\!=\!M^{-\ell}.$ The corresponding new MLMC estimator is
\begin{equation}
\widehat{Y} \coloneqq \frac{1}{N_0}\sum_{n=1}^{N_0} \varphi(\hX_0^{(n,0)})+ \sum_{\ell=1}^L\left\{ \frac{1}{N_\ell}\sum_{n=1}^{N_\ell} \left( \varphi(\hX_0^{(n,\ell)}) -\varphi(\hX_0^{(n,\ell-1)}) \right)\right\},
\label{MLMCT est}
\end{equation}
where $\hX_0^{(n,\ell)}$ is the terminal value of the $n$th numerical path through time interval $[-T_\ell,0]$ using adaptive function $h^\delta$ with $\delta=M^{-\ell}.$
Figure \ref{figAlgo} and Algorithm \ref{AlgoEuler} illustrate the detailed implementation of a single adaptive MLMC sample using a non-nested adaptive timestep on level $\ell$ with $M=2.$

\begin{figure}[p]
\begin{center}
\includegraphics[width=\textwidth]{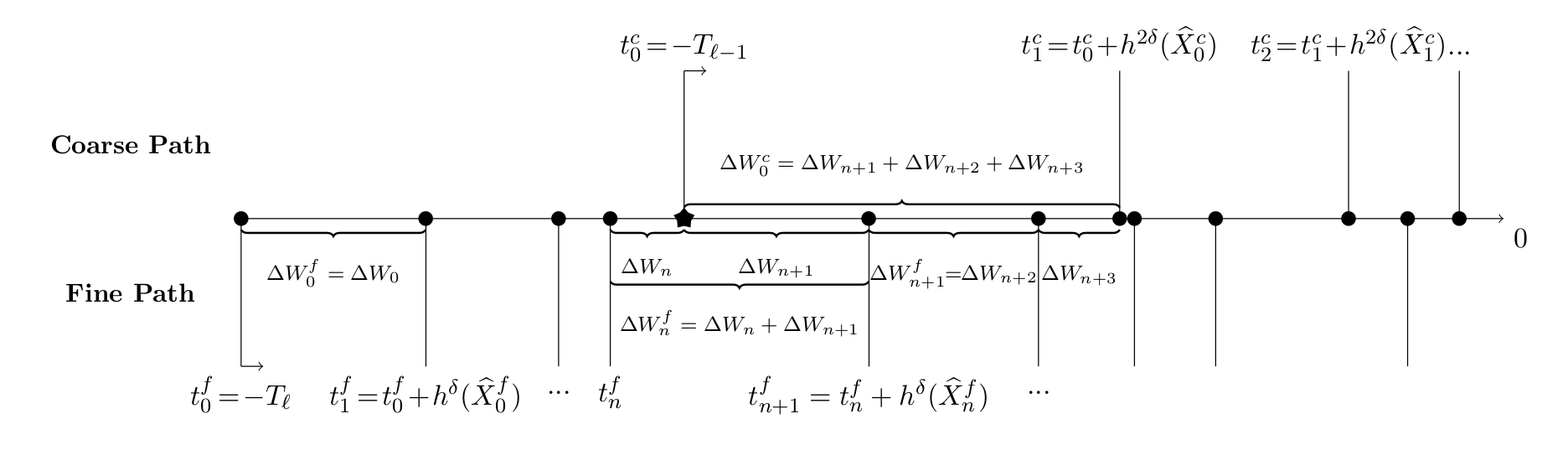}
\end{center}
\caption{Algorithm of adaptive MLMC for infinite interval}
\label{figAlgo}
\end{figure}

\begin{algorithm}{p}
\SetAlgoLined
   $t\coloneqq-T_\ell;\ t^c\coloneqq-T_{\ell-1};\ t^f\coloneqq-T_{\ell}$\;
   $h^c\coloneqq0;\ h^f\coloneqq0$\;
   $\Delta W^c\coloneqq 0;\ \Delta W^f\coloneqq 0$\;
   $\hX^c=x_0;\ \hX^f=x_0$\;
 \While{$t<0$}{
 
 $t_{old}\coloneqq t$\;
 $t\coloneqq \min(t^c,t^f)$\;
 $\Delta W\coloneqq N(0,t-t_{old})$\;
 $\Delta W^c\coloneqq \Delta W^c+\Delta W$\;
 \If{$t=-T_{\ell-1}$}{
 $\Delta W^c \coloneqq 0 $\;     
 }
 $\Delta W^f\coloneqq \Delta W^f+\Delta W$\;
 \If{$t=t^c$}{
  update coarse path $\hX^c$ using $h^c$ and $\Delta W^c$\;
  compute new adapted coarse path timestep $h^c=h^{2\delta}(\hX^c)$\;
  $h^c\coloneqq\min(h^c,-t^c)$\;
  $t^c\coloneqq t^c+h^c$\;
  $\Delta W^c\coloneqq 0$\;
 }
 \If{$t=t^f$}{
  update fine path $\hX^f$ using $h^f$ and $\Delta W^f$\;
  compute new adapted fine path timestep $h^f=h^{\delta}(\hX^f)$\;
  $h^f\coloneqq\min(h^f,-t^f)$\;
  $t^f\coloneqq t^f+h^f$\;
  $\Delta W^f\coloneqq 0$\;
 }

 }
 \KwResult{$\hX^f-\hX^c$}
 \caption{Outline of the algorithm for a single adaptive MLMC sample for scalar SDE on level $\ell$ in time interval $[-T_\ell,0].$}
\label{AlgoEuler}
\end{algorithm}

\subsection{Numerical analysis}

First, we state the exponential convergence to the invariant measure of the original SDEs, which can help us to measure the approximation error caused by truncating the infinite time interval.

\begin{lemma}[Exponential convergence]
If the SDE satisfies Assumption \ref{assp:linear_growth} and $\varphi$ satisfies the Lipschitz condition: there exists a constant $\kappa>0$ such that
\begin{equation}
\|\varphi(x)-\varphi(y)\|\leq \kappa\|x-y\|,
\label{eq: phi Lipschitz}
\end{equation}
 then this SDE is ergodic and has a unique invariant measure $\pi$ and there exist constants $\mu^*,\,\lambda^*>0$ such that
\begin{equation}
\left|\EE\left[\varphi(X_t)-\pi(\varphi)\right]\right|\leq \mu^*(1+\|x_0\|^2)\,\e^{-\lambda^*t}.
\label{eq: weak ergodic SDE convergence}
\end{equation} 
If the SDE additionally satisfies Assumption \ref{assp:ContractionLipschitz}, then there exists a constant $\mu>0$ depending on $x_0,$ $\kappa$ and $C_1$ in Lemma \ref{lemma:SDE_stability} such that
\begin{equation}
\left|\EE\left[\varphi(X_t)-\pi(\varphi)\right]\right|\leq \mu\,\e^{-\lambda t}.
\label{eq: ergodic SDE convergence}
\end{equation}
\label{lemma: Ergodicity SDE}
\end{lemma}
\vspace{-1.2cm}
\begin{proof}
The proof of existence and uniqueness of the invariant measure and convergence (\ref{eq: weak ergodic SDE convergence}) is given in Theorem 4.4 in \cite{MSH02} and Theorem 6.1 in \cite{MT93b}.

Next, we can define a new random variable $Y_0$ which follows the invariant measure $\pi,$ then the solution $Y_t$ to the SDE with the initial value $Y_0$ will also follows the invariant measure for any $t\!>\!0.$ Therefore, by Lipschitz property of $\varphi$ and Lemma \ref{lemma:SDE_stability} and \ref{lemma:SDE_contractivity}, there exists a constant $\mu\!>\!0$ such that
\begin{eqnarray*}
\left|\EE\left[\varphi(X_t)-\pi(\varphi)\right]\right| & = & \left|\EE\left[\varphi(X_t)-\varphi(Y_t)\right]\right|\ \leq\  \kappa\, \EE\left[\|X_t-Y_t\|\right]\nonumber\\
&\leq & \kappa\, \EE\left[\|X_0-Y_0\|\right]\e^{-\lambda t}\ \leq\  \kappa\,(\|x_0\|+C_1)\, \e^{-\lambda t}\coloneqq \mu \e^{-\lambda t}.
\end{eqnarray*}
\end{proof}
Note that $\lambda$ is easier to estimate than $\lambda^*$ through Assumption \ref{assp:Lipschitz_diff}.

\begin{lemma}[Variance of MLMC corrections for bounded volatility] If $\varphi$ satisfies the Lipschitz condition \eqref{eq: phi Lipschitz}, the SDE satisfies Assumption \ref{assp:ContractionLipschitz} and the timestep function $h^\delta$ satisfies Assumption 
\ref{assp:timestep_delta} with $\delta=M^{-\ell}$ for each level, then for each level $\ell,$ there exist constants $c_1$ and $c_2$ such that the variance of correction $V_{\ell}:=\mathbb{V}\left[\varphi(\hX_0^{\ell}) -\varphi(\hX_0^{\ell-1})\right]$ satisfies
\begin{equation}
V_\ell\leq c_1\, M^{-\ell}+c_2\, \e^{-2\lambda T_{\ell-1}}.
\end{equation} 
\label{lemma: strong error}
\end{lemma}
\vspace{-1.3cm}
\begin{proof}
Lipschitz condition \eqref{eq: phi Lipschitz} implies 
\[V_\ell \leq \EE\left[\left|\varphi(\hX_0^{\ell}) -\varphi(\hX_0^{\ell-1})\right|^2\right]\leq 
\kappa^2\, \EE\left[\left\|\hX_0^{\ell} - \hX_0^{\ell-1}\right\|^2\right].
\]
$\hX^{\ell}$ and $\hX^{\ell-1}$ share the same driving Brownian motion from $-T_{\ell-1}$ to $0.$ We can define the corresponding solution to the SDE (\ref{SDE}) starting from $x_0$ and driven by the same Brownian motion as $\hX^{\ell-1}$ through time interval $[-T_{\ell-1},0]$ by $X^c,$ and the solution starting from $x_0$ driven by the same Brownian motion as $\hX^{\ell}$ through time interval $[-T_{\ell},0]$ by $X^f.$

Then, by Jensen's inequality, we obtain that
\[\EE\left[\left\|\hX_{0}^{\ell} - \hX_{0}^{\ell-1}\right\|^2\right]\leq 3\,
(E_1+E_2+E_3),\]
where
\begin{eqnarray*}
E_1 &=& \EE\left[\left\|X^c_{0} - \hX_{0}^{\ell-1}\right\|^2\right],\\
E_2 &=& \EE\left[\left\|\hX_{0}^{\ell} - X^f_{0}\right\|^2\right],\\
E_3 &=& \EE\left[\left\|X^f_{0}- X^c_{0}\right\|^2\right].
\end{eqnarray*}
Theorem \ref{thm:convergence_order} implies that there exist a constant $C_2$ which does not depend on $T_\ell$ such that
\[E_1 \leq C_2 M^{-(\ell-1)},\ E_2 \leq C_2 M^{-\ell},\]
and Lemma \ref{lemma:SDE_stability} and Lemma \ref{lemma:SDE_contractivity}  imply that there exists a constant $C$ depending on $x_0$ and $C_2$ in Lemma \ref{lemma:SDE_stability} such that
\begin{eqnarray*}
E_3 &\leq& \EE\left[\|X^f_{-T_{\ell-1}}-x_0\|^2\right]\e^{-2\lambda T_{\ell-1}}\\
&\leq& 2\left(\EE\left[\|x_0\|^2\right]+\EE\left[\|X^f_{-T_{\ell-1}}\|^2\right]\right)\e^{-2\lambda T_{\ell-1}}\ \leq\ C\, \e^{-2\lambda T_{\ell-1}}.
\end{eqnarray*}
Finally, we obtain the desired result:
\[
V_\ell \leq 3\kappa^2\left(C_2 M^{-\ell+1}+ C_2 M^{-\ell}+C\e^{-2\lambda T_{\ell-1}}\right)\coloneqq c_1\, M^{-\ell}+c_2\, \e^{-2\lambda T_{\ell-1}}.
\]
\end{proof}

Note that if we set
\begin{equation}
T_\ell = (\ell\!+\!1)\log\!M / 2\lambda,
\label{eq:Tell}
\end{equation}
then
\begin{equation}
V_\ell \leq (c_1+c_2) M^{-\ell},
\label{eq:V_bound}
\end{equation}
which has the same order of magnitude as the variance bound 
for the standard finite time interval MLMC considered in Part I,
\cite{Part1}.

We define the computational cost of a path simulation to be 
equal to the number of timesteps. Hence, due to Lemma 
\ref{lemma: BoundedTimestep} and (\ref{eq:Tell}), there exists a 
constant $C_0$ such that the expected cost of a single MLMC sample
on level $\ell$ is bounded by $C_0 (\ell\!+\!1) M^\ell$.  Given this, 
we obtain the following theorem for the complexity of the MLMC 
algorithm to achieve a specified Mean Square Error accuracy.

\begin{theorem}[MLMC for invariant measure]If $\varphi$ satisfies the Lipschitz condition \eqref{eq: phi Lipschitz}, the SDE satisfies Assumption \ref{assp:ContractionLipschitz} and the timestep function $h^\delta$ satisfies Assumption 
\ref{assp:timestep_delta} with $\delta\!=\!M^{-\ell}$ for each level, then by choosing suitable values for $L$ and $T_\ell,\,N_\ell$ for each level $\ell,$ there exists a constant $c_3$ such that the MLMC estimator (\ref{MLMCT est}) has a mean square error (MSE) with bound
\[\EE\left[(\widehat{Y}-\pi(\varphi))^2\right]\leq \varepsilon^2,\]   
and an expected computational cost $C$ with bound
\[
C \leq c_3\, \varepsilon^{-2}|\log \varepsilon|^3.
\]
\label{thm: MLMC in T}
\end{theorem}
\vspace{-1.3cm}
\begin{proof}
By Jensen's inequality, the mean square error can be decomposed into three parts:
\begin{eqnarray*}
\EE\left[(\widehat{Y}-\pi(\varphi))^2\right]\!\!&\!\!=\!\!&\!\! \mathbb{V}\left[\widehat{Y}\right]+ \left|\EE\left[\widehat{Y}\right]-\pi(\varphi)\right|^2\\
\!\!&\!\!\leq\!\!&\!\! \mathbb{V}\left[\widehat{Y}\right]+ 2\left| \EE\left[\widehat{Y}\right]\!-\EE\left[\varphi(X_{T_L})\right]\right|^2\!\!+ 2\left|\EE\left[\varphi(X_{T_L})\right]-\pi(\varphi)\right|^2
\end{eqnarray*}
which enables us to achieve the MSE bound by bounding each part 
by $\varepsilon^2/3$.

Lemma \ref{lemma: Ergodicity SDE} implies that
\[
2\left|\EE\left[\varphi(X_{T_L})\right]-\pi(\varphi)\right|^2
\ \leq\  2\,\mu^2\e^{-2\lambda T_L}
\ \leq\ \frac{\varepsilon^2}{3}
\]
provided we ensure that
\[
T_L \geq \frac{|\log \varepsilon|}{\lambda }+\frac{\log(6\mu^2)}{2\lambda }.
\]
If we set $T_\ell$ according to (\ref{eq:Tell}),
this is achieved by requiring
\begin{equation}
L\geq \left\lfloor\frac{2|\log \varepsilon|}{\log M}+\frac{\log(6 \mu^2)}{\log M}\right\rfloor.
\label{eq:ineq1}
\end{equation}
By Theorem \ref{thm:convergence_order} and the Lipschitz property \eqref{eq: phi Lipschitz} of $\varphi$, there exists a constant $C_2$ such that
\begin{eqnarray*}
2\left| \EE\left[\widehat{Y}\right]-\EE\left[\varphi(X_{T_L})\right]\right|^2\!\!&\!=\!&\!2\left| \EE\left[\varphi(\hX_{T_L}^L)-\varphi(X_{T_L})\right]\right|^2\\
\!&\!\leq\!&\! 2\kappa^2\, \EE\left[\|\hX_{T_L}^L-X_{T_L}\|^2\right]
\ \leq\ 2\kappa^2 C_2M^{-L}
\ \leq\ \frac{\varepsilon^2}{3},
\end{eqnarray*}
provided
\begin{equation}
L\geq\left\lfloor\frac{2|\log \varepsilon|}{\log M}+\frac{\log(6\kappa^2C_2)}{\log M}\right\rfloor+1.
\label{eq:ineq2}
\end{equation}
Therefore, combining the requirements (\ref{eq:ineq1}) and (\ref{eq:ineq2}),
we choose to define
\begin{equation}
L=\left\lfloor\frac{2|\log \varepsilon|}{\log M}+\frac{\log\left(6 \max(\mu^2,\kappa^2 C_2)\right)}{\log M}\right\rfloor+1,
\label{eq: L}
\end{equation}
giving $L=O(|\log \varepsilon|)$ as $\varepsilon\rightarrow 0$.

Next, we need to choose the number of samples $N_\ell$ for each level.
We aim to minimize the total expected computational cost, which is bounded by
\[
C \leq C_0 \sum_{\ell=0}^L N_\ell (\ell\!+\!1) M^\ell,
\]
while at the same time ensuring that the total variance satisfies the bound
\[
\mathbb{V}\left[\widehat{Y}\right] 
\ = \ \sum_{\ell=0}^L N^{-1}_\ell V_\ell
\ \leq \ (c_1\!+\!c_2) \sum_{\ell=0}^L N^{-1}_\ell M^{-\ell}
\ \leq \  \frac{\varepsilon^2}{3}.
\]

Using a Lagrange multiplier, it is found that the optimal solution to the 
constrained optimization problem
\[
\min_{N_\ell}\   C_0 \sum_{\ell=0}^L N_\ell (\ell\!+\!1) M^\ell
\quad s.t. \quad 
(c_1\!+\!c_2) \sum_{\ell=0}^L N^{-1}_\ell M^{-\ell} \leq  \frac{\varepsilon^2}{3}.
\]
when the $N_\ell$ are treated as real variables is
\[
N_\ell = 3\, (c_1\!+\!c_2)\, \frac{M^{-\ell}}{\sqrt{\ell\!+\!1}}\ \varepsilon^{-2}\sum_{\ell'=0}^L\sqrt{\ell'\!+\!1}.
\]
Rounding this up to an integer by defining
\[
N_\ell = \left\lfloor 3\, (c_1\!+\!c_2)\, \frac{M^{-\ell}}{\sqrt{\ell\!+\!1}}\ \varepsilon^{-2}\sum_{\ell'=0}^L\sqrt{\ell'\!+\!1} \right\rfloor + 1.
\] 
we ensure that the required variance bound is satisfied.
The resulting cost is then bounded by
\[
C \ \leq\ 3\, C_0\, (c_1\!+\!c_2)\, \varepsilon^{-2}
\left(\sum_{\ell=0}^L \sqrt{\ell\!+\!1}\right)^2
+ C_0 \sum_{\ell=0}^L (\ell\!+\!1) M^\ell.
\]
Since
\[
\sum_{\ell=0}^L \sqrt{\ell\!+\!1}
\ \leq\  \int_0^{L+1} \! \sqrt{x\!+\!1}\ \D x
\ \leq\ \fracs{2}{3} (L\!+\!2)^{3/2}
\ = \ O(|\log \varepsilon|^{3/2}),
\]
and
\[
\sum_{\ell=0}^L (\ell\!+\!1) M^\ell
\ \leq\ (L\!+\!1)^2 M^L
\ = \ O(\varepsilon^{-2}|\log \varepsilon|^2),
\]
we obtain the desired final result that there 
exists a constant $c_3$ such that
\[
C \leq c_3 \, \varepsilon^{-2}|\log \varepsilon|^3.
\]
\end{proof}

For Langevin SDEs, the computational cost can be reduced to 
$O(\varepsilon^{-2}).$

\begin{theorem}[Langevin SDEs] If $\varphi$ satisfies the Lipschitz condition \eqref{eq: phi Lipschitz}, and for the SDE, $m\!=\!d$, $g\equiv I_m$, $f$ satisfies 
Assumption \ref{assp:enhanced_Lipschitz}, 
and the timestep function $h^\delta$ satisfies Assumption 
\ref{assp:timestep_delta} with $\delta=M^{-\ell}$ for each level, then for each level $\ell,$ there exist constants $c_1$ and $c_2$ such that 
\begin{equation}
V_\ell\leq c_1\, M^{-2\ell}+c_2\, \e^{-2\lambda T_{\ell-1}}.
\label{eq: variance Langevin}
\end{equation} 
Furthermore, by choosing suitable $T_\ell$ and $N_\ell$ for each level 
$\ell$ in the MLMC estimator (\ref{MLMCT est}), one can achieve the 
MSE bound $\varepsilon^2$ at an expected computational cost bounded by
\[
C\leq c_3\,\varepsilon^{-2},
\] 
for some constant $c_3\!>\!0$.
\label{thm: strong error 2}
\end{theorem}

\begin{proof}
Following a similar argument to the proof of Lemma \ref{lemma: strong error}, 
Theorem \ref{thm:convergence_order2} implies 
$
V_\ell\leq c_1\, M^{-2\ell}+c_2\, \e^{-2\lambda T_{\ell-1}},
$
and by choosing $T_\ell$ to be
\begin{equation}
T_\ell = (\ell\!+\!1) \log\! M / \lambda,
\label{eq:Tell2}
\end{equation}
we obtain $V_\ell\leq (c_1\!+\!c_2) M^{-2\ell}.$
The computational cost of a single MLMC sample on
level $\ell$ satisfies
\[
C_\ell \ \leq\ C_0 (\ell\!+\!1) M^\ell
      \ \leq\ C\, M^{(1+\epsilon)\ell}
\] 
for any $0\!<\!\epsilon\!\ll\! 1$ and some $C\!>\!0$.
Therefore, the standard MLMC Theorem 1 in \cite{giles15} 
is applicable with $\gamma\!<\!\beta$, giving an 
$O(\varepsilon^{-2})$ complexity.
\end{proof}
Note that the choice of $T_\ell$ \eqref{eq:Tell2} for the Langevin equation is different from \eqref{eq:Tell} for SDEs with bounded volatility. In other words, the strong convergence result and the contractive convergence rate $\lambda$ together determine $T_\ell$.  The difference in the variance convergence rate also
affects the choice of $M$.  Based on the analysis in \cite{giles08}, 
the optimal $M$ for SDEs with general $g$ is in the range $4-8$, while in 
the Langevin case the optimal $M$ is around $2$.

\section{Numerical Experiment}

In this section we present numerical results for the following scalar SDE:
\[
\D X_t=\left(-X_t-X_t^3\right)\D t + \D W_t,
\]
which satisfies the dissipativity condition (\ref{eq:onesided_growth}) and the contractive condition \eqref{eq:onesided_Lipschitz2}. Our interest is to compute $\pi(\varphi)$ where $\varphi(x)=\|x\|$ satisfying Lipschitz condition. 

Since the probability density function $\pi$ is 
\[\frac{\exp(-x^2-\frac{1}{2}x^4)}{\int_{-\infty}^{\infty}\exp(-x^2-\frac{1}{2}x^4)\,\D x},\] 
we can use numerical integration to calculate an approximate value:
$\varphi(\pi) \approx 0.44115$ 
with accuracy $10^{-5},$ and use this value as a benchmark for our numerical tests.
 
\begin{figure}[t!]
\begin{center}
\includegraphics[width=0.8\textwidth]{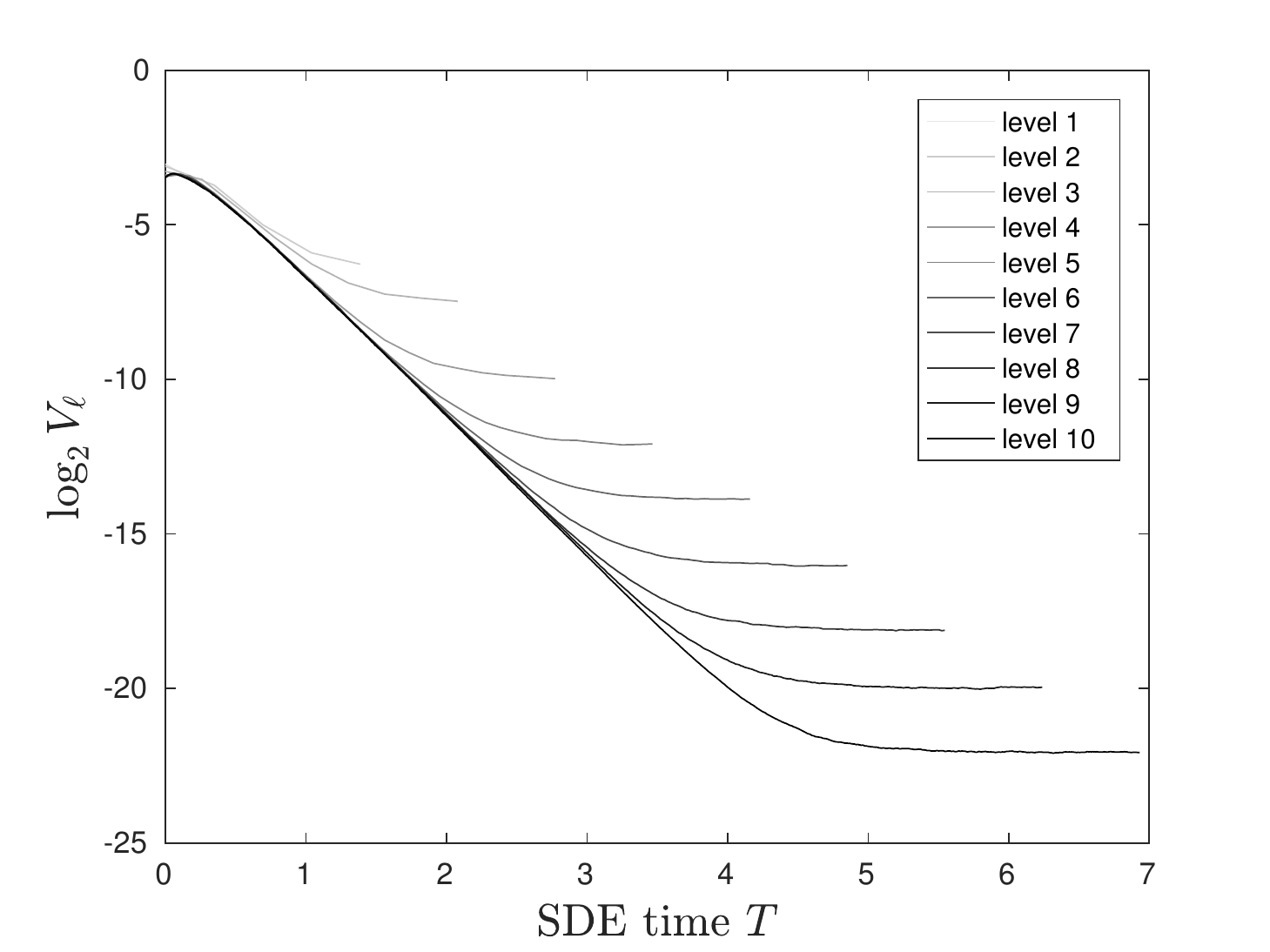}
\end{center}
\caption{Variance of corrections on each level $\ell$}
\label{figVar}
\end{figure}

Following condition (\ref{eq:timestep}) we can set $x_0\!=\!0$, $\hmax\!=\!1$, $M\!=\!2$ and choose the adaptive function $h,\ h^\delta$ to be 
\[h(x)=\frac{\max(1, |x|)}{\max(1, |x+x^3|)},\ \ h^{\delta}(x) = 2^{-\ell} h(x). \]

Next we need to determine $T_\ell$ for each level. Linear perturbations to the SDE
satisfy the ODE:
\[
\D Y_t= - \left(1+3\, X_t^2\right) Y_t\, \D t,
\]
and therefore $\lambda\!\geq\! 1$.  Hence we choose to use
\[
T_\ell = (\ell\!+\!1) \log 2
\]
to ensure that the truncation error is acceptably small.

The variance result (\ref{eq: variance Langevin}) for the Langevin equation is illustrated in Figure \ref{figVar}. The exponential term dominates the variance initially, but as $T_\ell$ increase, 
the $M^{-2\ell}$ term eventually becomes the major part of the variance.

\begin{figure}[t!]
\begin{center}
\includegraphics[width=\textwidth,height=\textwidth]{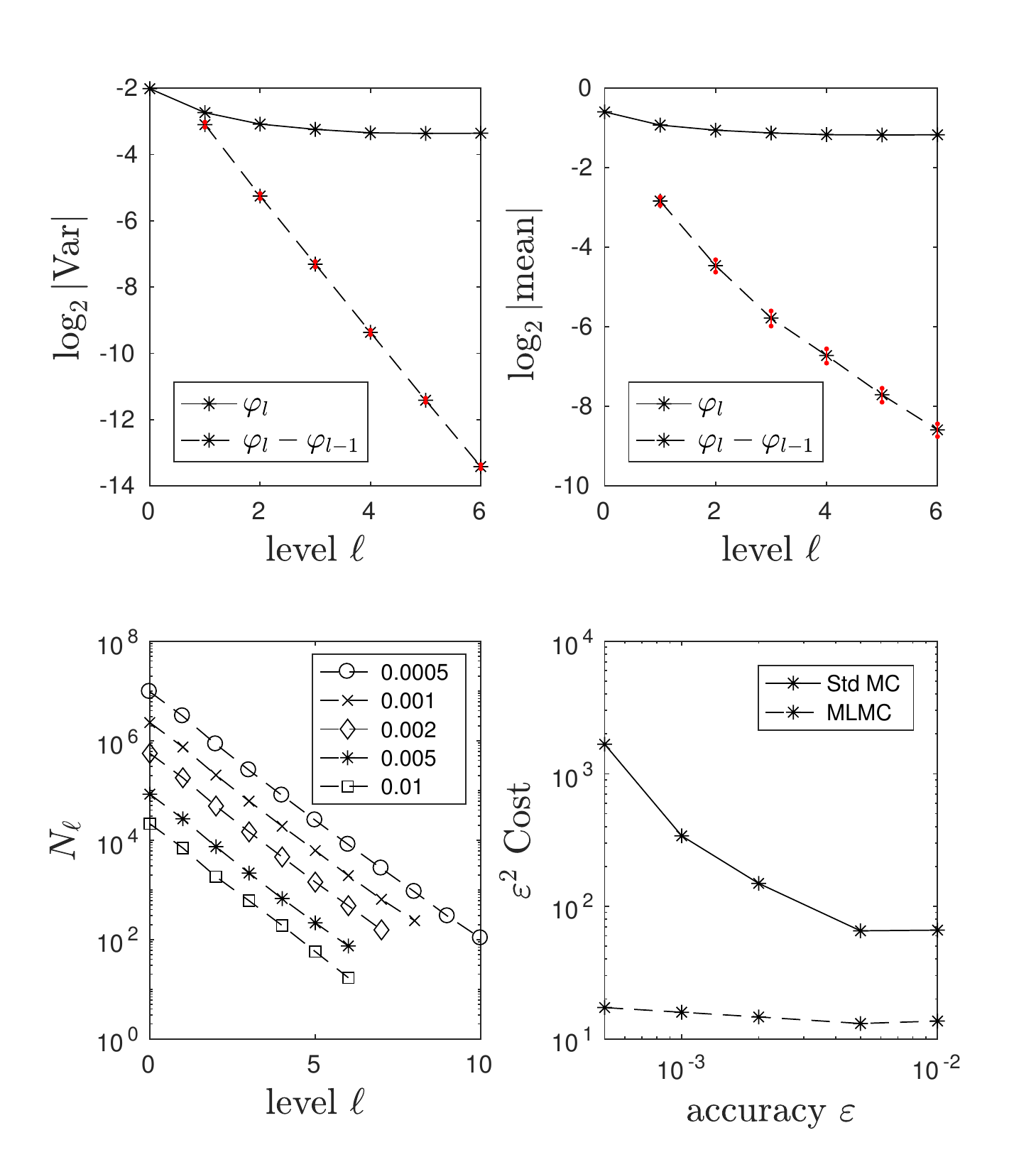}
\end{center}
\caption{Adaptive MLMC for invariant measure}
\label{fig2}
\end{figure}

Figure \ref{fig2} presents the MLMC results. The top right plot shows first order convergence for the weak error and the top left plot shows second order convergence for the multilevel correction variance. Hence the computational
cost for RMS accuracy $\varepsilon$ is $O(\varepsilon^{-2})$ which is verified in the bottom 
right plot, while the bottom left plot shows the number of MLMC samples on each level as a 
function of the target accuracy.

\clearpage

\section{Proofs}

\label{sec:proofs}

This section has the proofs of the three main theorems in this paper,
one on stability, and two on the order of strong convergence.

\subsection{Preliminaries}
In this subsection, we introduce some inequalities and results we use frequently in the following sections.
\subsubsection{Young inequality}
For any $\zeta\!>\!0$ and $P,Q\!>\!1$ satisfying $\frac{1}{P}\!+\!\frac{1}{Q}=1,$ the following inequality holds for any $A,B\!>\!0,$
\begin{equation*}
A\, B\,=\,A\zeta\ \frac{B}{\zeta}\,\leq \frac{A^P\zeta^P}{P}+\frac{B^Q}{Q\, \zeta^Q}.
\end{equation*}
In this paper, we use two particular cases. First, we take $P\!=\!Q\!=\!2$, and 
$\zeta^2\!=\!2\xi \!>\! 0$ and get
\begin{equation}
A\, B\leq \xi A^2 + \frac{B^2}{4\,\xi}.
\label{Young inequality 2}
\end{equation}
Second, for any $p\geq2$ and $\xi>0,$ we take $P\!=\!\frac{p}{p-2},$ $Q\!=\!\frac{p}{2},$ 
$A\!=\!a^{p-2},$ $B\!=\!b^2$ and $\zeta \!=\! \xi^{(p-2)/p},$ and get   
\begin{equation}
a^{p-2}b^2\leq \frac{(p\!-\!2)\,\xi}{p}\,a^p+\frac{2}{p\,\xi^{(p-2)/2}}\,b^p.
\label{Young inequality}
\end{equation}
When using these in proofs, we often keep $\xi$ arbitrary initially and choose it later to make one term sufficiently small, as needed.

\subsubsection{Jensen inequality} 
One variant of the Jensen inequality is
\begin{equation}
\left( \sum_k \e^{\lambda t_k}h_k\, u_k \right)^p \!\leq\!
\left(\sum_k \e^{\lambda t_k}h_k \right)^{p-1} \sum_k \e^{\lambda t_k} h_k u^p_k,
\label{Jensen Dis}
\end{equation}
where $p\geq 1$ and $h_k,\ u_k\geq 0.$ Its continuous version is 
\begin{equation}
\left|\int_0^t \phi(s)\,\e^{\gamma s}\,\mathrm{d}s\right|^p\ \leq\ \left(\int_0^t\e^{\gamma s}\,\mathrm{d}s \right)^{p-1}\int_0^t |\phi(s)|^p\e^{\gamma s}\,\mathrm{d}s.
\label{Jensen exp}
\end{equation}

\subsubsection{Exponentially weighted supremum} For simplicity, 
for $\alpha\!>\!0$, we can define $\hM_t^{\alpha,p}\triangleq \sup_{0\leq s\leq t} \e^{\alpha p s}\|\hX_s\|^p,$ and then by 
Young's inequality (\ref{Young inequality 2}), for any $\xi\!>\!0,$
\begin{equation}
\hM_t^{\alpha,p/2}\leq \xi\,\hM_t^{\alpha,p}+\frac{1}{4\,\xi}.
\label{Simple 1}
\end{equation}
We can also define $\bM_t^{\alpha,p}\triangleq \sup_{0\leq s\leq t} \e^{\alpha p s}\|\bX_s\|^p,$ 
which implies
\begin{equation}
\bM_t^{\alpha,p}\leq \e^{\alpha p \hmax}\,\hM_t^{\alpha,p},
\label{Simple 2}
\end{equation}
since $\bX_s\!=\!\hX_\sd$ and $|s\!-\!\sd|\leq \hmax,$
and
\begin{eqnarray}
\int_0^t \e^{\gamma ps/2}\|\bX_s\|^{p/2} \D s 
&\leq& \bM_t^{\alpha,p/2}\int_0^t \e^{(\gamma-\alpha)ps/2} \D s 
\nonumber \\&\leq& 
\frac{2\e^{(\gamma-\alpha)pt/2}}{p\,(\gamma-\alpha)}\,\e^{\alpha p \hmax/2}\hM_t^{\alpha,p/2}.
\label{Simple 3} 
\end{eqnarray}
provided $\gamma\!>\!\alpha\!>\!0$. 

\subsection{Theorem \ref{thm:stability}}

\begin{proof}
By theorem 1 in Part I \cite{Part1}, we know $T$ is almost surely attainable. Therefore we can directly analyse our discretization scheme without the $K$ truncation which was used in that paper. The proof proceeds in three steps.  First, we derive an upper bound 
for $\e^{\alpha pt}\|\hX_t\|^p$. Second, we show that the moments
$\EE[ \hM^{\alpha,p}_t ]$ and $\EE[ \bM^{\alpha,p}_t ]$ are each bounded by $C_p\e^{\alpha p t}$ where $C_{p}$ is a constant  which only depends on $p$, $x_0$, $\hmax$ and the constants $\alpha, \beta$ in
Assumption \ref{assp:timestep}. Finally, we get the uniform bound for $\EE[ \|\hX_t\|^p ]$ and $\EE[ \|\bX_t\|^p ].$

The proof is given for $p\!\geq\! 4$; the result for 
$0\!<\! p \!<\! 4$ follows from H{\"o}lder's inequality.

\noindent
\textbf{{Step 1: }}
If we define $\phi(x)\triangleq x\!+\!h(x)f(x)$,
then (\ref{eq:dis_SDE}) gives 
\begin{eqnarray*}
\|\hX_\tnp\|^2  & = & \|\hX_\tn\|^2 
+ 2\, h_n\left( \langle\hX_\tn,f(\hX_\tn)\rangle + \halfs h_n \|f(\hX_\tn)\|^2 \right)
\\ &&
+\, 2\,\langle \phi(\hX_\tn), g(\hX_\tn)\,\DW_n \rangle
+ \|g(\hX_\tn)\, \DW_n\|^2.
\end{eqnarray*}
Using condition (\ref{eq:timestep}) for $h$ then gives
\begin{eqnarray*}
\|\hX_\tnp\|^2 & \leq & 
\|\hX_\tn\|^2 \ -\  2\,\alpha\|\, \hX_\tn\|^2h_n \ +\ 2\,\beta\, h_n 
\nonumber\\&& 
+\, 2\,\langle \phi(\hX_\tn), g(\hX_\tn)\,\DW_n \rangle
+ \|g(\hX_\tn)\, \DW_n\|^2.
\end{eqnarray*}
Since $1\!-\!2\alpha h_n\leq \e^{-2\alpha h_n}$ and $g$ and $h$ are both bounded, 
we multiply by $\e^{2\alpha \tnp}$ on both sides to obtain
\begin{eqnarray}
\e^{2\alpha \tnp}\|\hX_\tnp\|^2 & \!\!\leq\!\! &
\e^{2\alpha t_n}\|\hX_\tn\|^2 \! + 2\e^{2\alpha (\tn+\hmax)}\,\beta\, h_n 
+ \e^{2\alpha (\tn+\hmax)} \beta\|\, \DW_n\|^2
\nonumber\\ && 
+\, 2\,\e^{2\alpha \tnp}\langle \phi(\hX_\tn), g(\hX_\tn)\,\DW_n \rangle.
\label{eq:dis_one_step}
\end{eqnarray}
Similarly, for the partial timestep from $\td$ to $t$, 
since $(t\!-\!\td)\leq h_{n_t},$
\begin{eqnarray}
 \langle\hX_\td,f(\hX_\td)\rangle + \halfs (t-\td) \|f(\hX_\td)\|^2 \leq -\alpha\|\hX_\td\|^2+\beta, 
\label{eq:partial}
\end{eqnarray}
and therefore we obtain
\begin{eqnarray}
\e^{2\alpha t}\|\hX_t\|^2 \!\!\!\!& \leq & \!\!\!\!
\e^{2\alpha \td}\|\hX_{\td}\|^2 \ +\ 2\e^{2\alpha (\td+\hmax)}\,\beta\, (t-\td) + \e^{2\alpha (\td+\hmax)} \beta\|\, W_t-W_\td\|^2
\nonumber\\&& 
+\, 2\,\e^{2\alpha t}\langle \phi(\hX_\td), g(\hX_\td)\,(W_t-W_\td) \rangle.
\label{eq:con_one_step}
\end{eqnarray}
Summing (\ref{eq:dis_one_step}) over multiple timesteps and then adding (\ref{eq:con_one_step}) gives
\begin{eqnarray*}
\e^{2\alpha t}\|\hX_t\|^2 & \leq & 
\|x_0\|^2 \ +\ 2\beta \e^{2\alpha \hmax}\!\left(\sum_{k=0}^{n_t-1} \e^{2\alpha t_{k}}h_k+\e^{2\alpha \td}(t-\td)\right)
\\[-0.1in] &&
\!\!\!\!\!\!\!\!\!\!\!\!\!\!\!\!\!\!\!\!\!\!\!+\, 2\sum_{k=0}^{n_t-1}\e^{2\alpha t_{k+1}}\langle \phi(\hX_{t_k}), g(\hX_{t_k})\DW_k \rangle)
+  \beta \e^{2\alpha \hmax}\sum_{k=0}^{n_t-1}\e^{2\alpha t_k}\|\DW_k\|^2
\\ &&
\!\!\!\!\!\!\!\!\!\!\!\!\!\!\!\!\!\!\!\!\!\!\!+\,  2 \e^{2\alpha t}\langle  \hX_\td \!+\! f(\hX_\td)\, (t\!-\!\td),\, g(\hX_\td)(W_t\!-\!W_\td) \rangle
+\, \beta \e^{2\alpha (\td+\hmax)} \|W_t\!-\!W_\td\|^2.
\end{eqnarray*}
Bounding the first summation using a Riemann integral, and re-writing the second 
as an It{\^o} integral, raising both sides to the power $p/2$ 
and using Jensen's inequality, we obtain
\begin{eqnarray}
\label{os3fs4}
\e^{\alpha pt}\|\hX_t\|^p & \!\!\leq\!\! & 
6^{p/2 - 1}\e^{\alpha p \hmax}\left\{ \rule{0in}{0.25in}
\|x_0\|^p \ +\  
\left(2\, \beta \int_0^t\e^{2\alpha s}\,\D s\right)^{p/2}\ 
\right. \nonumber \\ && 
\!\!\!\!\!\!\!\!\!\!\!\!\!\!\!\!\!\!\!\!\!\! 
+ \, \left|\, 2\! \int_0^\td \e^{2\alpha (\sd+h(\bX_s))}\! \langle \phi(\bX_s), g(\bX_s)\, \D W_s\rangle \,\right|^{p/2}
 + \left(\beta\sum_{k=0}^{n_t-1}\e^{2\alpha t_k}\|\DW_k \|^2 \! \right)^{p/2}
\nonumber \\ &&
\!\!\!\!\!\!\!\!\!\!\!\!\!\!\!\!\!\!\!\!\!\!
\left.+  \left|2\e^{2\alpha \td}\langle  \bX_t \!+\! f(\bX_t)\, (t\!-\!\td), g(\bX_t) (W_t\!-\!W_\td) \rangle \right|^{p/2}
\!\!
+\, \beta^{p/2} \e^{\alpha p \td}\|W_t\!-\!W_\td\|^p 
\rule{0in}{0.25in} \right\}.\nonumber \\
\end{eqnarray}


\noindent
\textbf{{Step 2: }} For any $0\!\leq\! t\!\leq\! T,$ we take the supremum on 
both sides of inequality (\ref{os3fs4}) and 
then take the expectation to obtain
\[
\EE\left[\hM_t^{\alpha,p}\right]=\EE\left[ \sup_{0\leq s\leq t}\e^{\alpha p s}\|\hX_s\|^p\right] \ \leq \ 
 6^{p/2 - 1}\e^{\alpha p \hmax}\left( I_1 + I_2  + I_3  + I_4 + I_5 \right), 
\]
where
\begin{eqnarray*}
I_1 &=& \|x_0\|^p +  \left(2\, \beta \int_0^t\e^{2\alpha s}\,\D s\right)^{p/2}, 
\\[0.05in]
I_2 &=& \EE\left[ \sup_{0\leq s\leq \td}\left| \,2\! \int_0^s \e^{2\alpha (\ud+h(\bX_u))}
\langle \phi(\bX_u), g(\bX_u)\,\D W_u \rangle \right|^{p/2} \right],
\end{eqnarray*}
\begin{eqnarray*}
I_3 &=& \EE\left[ \left(\beta \sum_{k=0}^{n_t-1}\e^{2\alpha t_k}\| \DW_k\|^2\right)^{p/2}\right],
\\[0.05in]
I_4 &=& \EE\left[ \sup_{0\leq s\leq t} 
\left| 2 \e^{2\alpha \sd}\langle \bX_s \!+\! f(\bX_s)\, (s\!-\!\sd), g(\bX_s)  (W_s\!-\!W_\sd) \rangle \right|^{p/2}\right],
\\[0.05in]
I_5 &=& \EE\left[\sup_{0\leq s\leq t} \beta^{p/2}\e^{\alpha p \sd} \|W_s\!-\!W_\sd\|^p \right].
\end{eqnarray*}
We now consider $I_1, I_2, I_3, I_4, I_5$ in turn.
\[
I_1 =\|x_0\|^p 
+ (2 \beta)^{p/2} \left(\frac{\e^{2\alpha t}-1}{2\alpha}\right)^{p/2}\leq \|x_0\|^p
+ (\beta/\alpha)^{p/2} \e^{\alpha pt}.
\]
By the Burkholder-Davis-Gundy inequality, there exist constants $C_p^1$ such that
\begin{eqnarray*}
I_2 &= & 
\EE\left[ \sup_{0\leq s\leq \td}\left| \,2\! \int_0^s \e^{2\alpha (\ud+h(\bX_u))}\langle \phi(\bX_u), g(\bX_u)\, \D W_u\rangle \right|^{p/2}\right]\\
&\leq&
 \EE\left[C_p^1 \left(\int_0^\td \e^{4\alpha u}\| \phi(\bX_u)^T g(\bX_u) \|^2 \, \D u \right)^{p/4} \right] .
\end{eqnarray*}
Due to condition (\ref{eq:timestep}), for $u\!<\!\td$ we have
\begin{eqnarray*}
\|\phi(\bX_u)\|^2 
   &=& \|\bX_u\|^2 + 2\, h(\bX_u)\, 
\left( \langle \bX_u, f(\bX_u)\rangle + \halfs\, h(\bX_u) \|f(\bX_u)\|^2 \rule{0in}{0.16in} \right)
\\ &\leq&
\|\bX_u\|^2+ 2 \, h(\bX_u) \, (-\alpha\|\bX_u\|^2 +\beta)
\\ &\leq& \|\bX_u\|^2 + 2\, \beta \hmax,
\end{eqnarray*}
and hence by Jensen's inequality and the boundedness condition (\ref{eq:g_bound}) of $g$ , we obtain
\[
\|\phi(\bX_u)^T g(\bX_u)\|^{p/2} \leq 2^{p/4-1}\beta^{p/4} \left(\rule{0in}{0.16in}
 \|\bX_u\|^{p/2} + (2\,\beta \hmax)^{p/4} 
\right).
\]

Therefore, using Jensen's inequality (\ref{Jensen exp}) with 
$\gamma\!=\!2\alpha$, followed by (\ref{Simple 3}) with 
$\gamma \!=\! (1+4/p)\alpha$ and then 
(\ref{Simple 1}) with $\xi\!=\!\e^{-\alpha p t/2}\zeta$, there exists 
a constant $C_p^2$ which is linearly dependent on $\zeta^{-1}$ such that
\begin{eqnarray*}
 I_2&\leq& \EE\left[ 
C_p^1 (\e^{2\alpha t}/(2\alpha))^{p/4-1}\int_0^t \e^{\alpha(p/2+2) u}\| \phi(\bX_u)^T g(\bX_u) \|^{p/2} \, \D u \right] 
\\ &\leq&  \EE\left[
C_p^1 (\e^{2\alpha t}/\alpha)^{p/4-1}\beta^{p/4} \int_0^t \e^{\alpha(p/2+2) u}\,\left( \|\bX_u\|^{p/2} + (2\,\beta \hmax)^{p/4} \right) \D u \right] 
\\
&  \leq &\EE\left[ \frac{C_p^1}{2}\left(\frac{\beta}{\alpha}\right)^{p/4} \e^{\alpha p (t+\hmax)/2} \hM_t^{\alpha,p/2}  \right] + C_p^1 \beta^{p/2}\left(\frac{2\hmax}{\alpha}\right)^{p/4}\frac{2 \e^{\alpha p t}}{p+4}\\
& \leq &  
\frac{C_p^1}{2}\left(\frac{\beta}{\alpha}\right)^{p/4} \e^{\alpha p \hmax/2}\, \zeta
\,\EE\left[ \hM_t^{\alpha,p} \right]+C_p^2\, \e^{\alpha p t}.
\label{os3fn8}
\end{eqnarray*}

Using Jensen inequality (\ref{Jensen Dis}), we obtain
\begin{eqnarray*}
I_3 &\leq& 
\beta^{p/2}\, \left(\int_0^t \e^{2\alpha s}\D s\right)^{p/2-1}
\EE\left[\ \sum_{k=0}^{n_t-1} h_k \,\e^{2\alpha t_k}\ \frac{\|\DW_k\|^p}{h_k^{p/2}}
\right]
\\ &\leq & \, c_p\left(\beta\int_0^t \e^{2\alpha s}\D s\right)^{p/2}
\leq c_p(\beta/2\alpha)^{p/2}\e^{\alpha p t},
\end{eqnarray*}
where we define
\[
c_p \triangleq \EE[\sup_{0\leq t\leq 1}\|W_t-W_0\|^p] < \infty.
\]
so that
$\EE[\|\DW_k\|^p] \leq c_p h_k^{p/2}$.

In considering $I_4$, we start by observing that for $t_k\!\leq\!s\!<\!t_{k+1}$ 
\begin{equation}
\EE\!\left[ \sup_{t_k\leq u \leq s} 
\!\! \|(W_u\!-\!W_{t_k}) \|^p\ |\ \mathcal{F}_{\sd}\right] 
\, =\,  c_p\, (s\!-\!\sd)^{p/2} 
\, \leq\,  c_p \hmax^{p/2-1}(s\!-\!\sd).
\label{eq:DW}
\end{equation}
In addition, using (\ref{eq:partial}) and following the same argument as for $I_2$, 
we have
\[
\|\bX_s\!+\! f(\bX_s) (s\!-\!\sd)\|^{p/2} \| g(\bX_s)\|^{p/2} \leq
2^{p/4-1}\beta^{p/4} \left(
 \|\bX_s\|^{p/2} + (2\,\beta \hmax)^{p/4} 
\right).
\]
Therefore, combining the estimation (\ref{eq:DW}), (\ref{Simple 3}) 
with $\gamma\!=\!2\alpha$ and (\ref{Simple 1}) with 
$\xi=\e^{-\alpha pt/2}\zeta$, there exists $C_p^3$ which is 
linearly dependent on $\zeta^{-1}$ such that
\begin{eqnarray*}
I_4 &\leq& 
2^{p/2}\, \EE\left[ \sup_{0\leq s\leq t} \e^{\alpha p \sd} \left|
\langle \bX_s\!+\! f(\bX_s) (s\!-\!\sd), g(\bX_s)\, (W_s\!-\!W_\sd) \rangle \right|^{p/2}\right]\\
&\leq& 2^{p/2}\, \EE\left[ \sup_{0\leq s\leq t} \e^{\alpha p \sd} \| \bX_s\!+\! f(\bX_s) (s\!-\!\sd)\|^{p/2} \|g(\bX_s)\|^{p/2}\| (W_s\!-\!W_\sd) \|^{p/2}\right]
\\ &\leq&
2^{3p/4-1}\beta^{p/4}\, \EE\!\left[\sum_{k=0}^{n_t-1} \! \e^{\alpha p t_k}
\left(\|\bX_{t_k}\|^{p/2}+(2\beta \hmax)^{p/4}\right)
\!\!\!\sup_{t_k\leq s< t_{k+1}} \!\!\! \|W_s\!-\!W_{\sd}\|^{p/2} \right.
\\&& \hspace{0.6in} \left. +\, \e^{\alpha p \td}
\left(\|\bX_{\td}\|^{p/2}+(2\beta \hmax)^{p/4}\right)
\!\sup_{\td\leq s< t} \! \| W_s\!-\!W_\sd\|^{p/2}\right]
\\ &\leq& 
2^{3p/4-1}\beta^{p/4}c_{p/2} \hmax^{p/4-1}\,
\EE\left[ \int_0^t \e^{\alpha p s}\left(\|\bX_s\|^{p/2}+(2\beta \hmax)^{p/4}\right)\D s \right] 
\\&\leq&
2^{3p/4}\beta^{p/4}c_{p/2} \hmax^{p/4-1}\, (p \alpha)^{-1}\e^{\alpha p \hmax/2}\, \zeta\,
\EE\left[\hM_t^{\alpha,p}\right]+C_p^3\e^{\alpha p t}.
\end{eqnarray*}
Similarly, again using the same definition for $c_p$, we have
\[
I_5 \leq 
c_p\, \beta^{p/2} \hmax^{p/2-1}  \e^{\alpha p t}/(\alpha p).
\]

Collecting together the bounds for $I_1, I_2, I_3, I_4, I_5$,
we conclude that we can choose $\zeta\!>\!0$ sufficiently small so that 
there exist constants $C^4_{p}$ and $C_p^5$ such that
\[
\EE\left[\hM_t^{\alpha,p}\right]
\leq \fracs{1}{2}\, \EE\left[ \hM_t^{\alpha,p}\right]
+ C_p^4 \|x_0\|^p + C_p^5\, \e^{\alpha p t},
\]
and hence
\[
\EE\left[\hM_t^{\alpha,p} \right]
\leq 2C_p^4 \,\|x_0\|^p+2C_p^5\, \e^{\alpha p t}.
\]

\noindent
\textbf{{Step 3: }}
Due to the definition of $\bM_t^{\alpha,p}$ and inequality (\ref{Simple 2}),
for any $t\geq 0,$
\begin{eqnarray*}
\EE\left[\|\bX_t\|^p\right] 
\ \leq\  \e^{-\alpha p t}\, \EE\left[\bM_t^{\alpha,p} \right]
&\leq& \e^{-\alpha p t}\, \e^{\alpha p \hmax} \, \EE\left[\hM_t^{\alpha,p} \right]
\\&\leq& \e^{\alpha p \hmax}(2\,C_p^4\|x_0\|^p +2C_p^5)
\ \triangleq\ C_p
\end{eqnarray*}
and similarly
\[
\EE\left[\|\hX_t\|^p\right] 
\ \leq\ \e^{-\alpha p t}\, \EE\left[\hM_t^{\alpha,p} \right]
\ \leq\ 2\,C_p^4\, \|x_0\|^p+2C_p^5
\ < \ C_p.
\]
\end{proof}

\subsection{Theorem \ref{thm:convergence_order}}

\begin{proof}
The approach which is followed is to bound the approximation
error $e_t\triangleq \hX_t - X_t$ by terms which depend on 
$\hX_s\!-\!\bX_s$, and then use local analysis within 
each timestep to bound these. The proof is for 
$2\leq\! p\!\leq\! p^*$; the result for $0\!<\! p \!<\! 2$ follows 
from H{\"o}lder's inequality.

We start by combining the original SDE \eqref{SDE} with (\ref{eq:approx_SDE})
to obtain
\[
e_t
\ =\ \int_0^t \left( f(\bX_s)-f(X_s)\right)\,\D s +\int_0^t\left( g(\bX_s)-g(X_s)\right)\,\D W_s,
\]
and then by It{\^o}'s formula and Young's inequality (\ref{Young inequality 2}), together with $e_0=0,$ and $\lambda,\ \eta$ as defined in Assumption \ref{assp:ContractionLipschitz}, we get
\begin{eqnarray*}
\e^{\lambda p t/2}\|e_t\|^p&\!\!\!\!\!\!\leq&\!\!\!\!\!\!
\int_0^t \frac{p\lambda}{2}\, \e^{\lambda p s/2}\|e_s\|^p\ \D s+\int_0^t p\langle e_s, f(\bX_s)\!-\!f(X_s)\rangle \e^{\lambda p s/2}\|e_s\|^{p-2} \,\D s \\
&&+\int_0^t \frac{p(p-1)}{2}\|g(\bX_s)\!-\!g(X_s)\|^2 \e^{\lambda p s/2}\|e_s\|^{p-2} \,\D s \\ 
&&+\int_0^t p\, \langle e_s, (g(\bX_s)\!-\!g(X_s))\e^{\lambda p s/2}\,\|e_s\|^{p-2}\,\D W_s \rangle\\
&\!\!\!\leq&\!\!\!
\int_0^t\frac{ p\lambda}{2}\, \e^{\lambda p s/2}\|e_s\|^p\ \D s+\!\!\! \int_0^t p\langle e_s, f(\hX_s)\!-\!f(X_s)\rangle \e^{\lambda p s/2}\,\|e_s\|^{p-2} \,\D s \\
&&- \int_0^t p \langle e_s, f(\hX_s)\!-\!f(\bX_s)\rangle \e^{\lambda p s/2}\,\|e_s\|^{p-2} \,\D s \\
&&\!\!\!\!+p\!\int_0^t\! \left(\frac{p-1}{2}+\frac{\lambda}{4\eta}\right)\|g(\hX_s)\!-\!g(X_s)\|^2\e^{\lambda p s/2}\,\|e_s\|^{p-2} \D s\\
&&\!\!\!\!+p\!\int_0^t\!\left(\frac{p-1}{2}+\frac{\eta(p-1)^2}{\lambda}\right)\!\|g(\hX_s)\!-\!g(\bX_s)\|^2\e^{\lambda p s/2}\,\|e_s\|^{p-2} \D s\\
&&\!\!\!\!+\int_0^t p\, \langle e_s, (g(\bX_s)\!-\!g(X_s))\e^{\lambda p s/2}\,\|e_s\|^{p-2}\,\D W_s \rangle
\end{eqnarray*}
Using the conditions in Assumption \ref{assp:ContractionLipschitz}, (\ref{eq:g_Lipschitz}) implies that
\[
\|g(\hX_s)-g(\bX_s)\|^2\leq \eta\|\hX_s-\bX_s\|^2.
\]
(\ref{eq:contraction_Lipschitz}) and (\ref{eq:g_Lipschitz}) imply that
\[
\langle e_s, f(\hX_s)\!-\!f(X_s)\rangle + \left(\frac{p-1}{2}+\frac{\lambda}{4\eta}\right)\|g(\hX_s)-g(X_s)\|^2\leq -\frac{3\lambda}{4}\, \|e_s\|^2.
\]
(\ref{eq:local_Lipschitz}) and Young inequality (\ref{Young inequality 2}) implies that
\begin{eqnarray*}
\left| \langle e_s, f(\hX_s)\!-\!f(\bX_s) \rangle \right|
  &\leq& \|e_s\| \, L(\hX_s,\bX_s) \, \| \hX_s\!-\!\bX_s \|
\\&\leq& \frac{\lambda}{8} \|e_s\|^2 + \frac{2}{\lambda} L(\hX_s,\bX_s)^2 \| \hX_s\!-\!\bX_s \|^2.
\end{eqnarray*}
where $L(x,y)\triangleq \gamma(\|x\|^q + \|y\|^q) + \mu$. Hence,
\begin{eqnarray*}
  \e^{\lambda p t/2}\|e_t\|^p &\!\leq\!& \!\!\, \int_0^t -\frac{p\lambda}{8}\, \e^{\lambda p s/2}\|e_s\|^p\ \D s\\
  &&+ \int_0^t p\,
\hat{L}(\hX_s,\bX_s)\|\hX_s\!-\!\bX_s\|^2\e^{\lambda p s/2}\|e_s\|^{p-2}\, \D s\\ 
&&+\int_0^t p\langle e_s, (g(\bX_s)\!-\!g(X_s))\e^{\lambda p s/2}\|e_s\|^{p-2}\,\D W_s \rangle,
\end{eqnarray*}
where $\hat{L}(x,y)=\frac{2}{\lambda}
L(x,y)^2 \!+\! \frac{(p-1)\eta}{2}+\frac{\eta^2(p-1)^2}{\lambda}.$
Young inequality (\ref{Young inequality}) implies
\begin{eqnarray*}
  \e^{\lambda p t/2}\|e_t\|^p\!\! &\!\leq\!& \! \int_0^t 2\left(\frac{8(p\!-\!2)}{p\lambda}\right)^{p/2-1}
\hat{L}(\hX_s,\bX_s)^{p/2}\e^{\lambda p s/2}\|\hX_s\!-\!\bX_s\|^p\,\D s  \\ 
&&  +\, \int_0^t p\langle e_s, (g(\bX_s)\!-\!g(X_s))\e^{\lambda p t/2}\|e_s\|^{p-2}\,\D W_s \rangle.
\end{eqnarray*}
Taking the expectation of each side yields
\[
\label{eq:halfway}
\EE\left[ \e^{\lambda p t/2}\|e_t\|^p\right] \leq
2\left(\frac{8(p\!-\!2)}{p\lambda}\right)^{p/2-1}\int_0^t \EE\left[\hat{L}(\hX_s,\bX_s)^{p/2}\|\hX_s\!-\!\bX_s\|^p\right] \e^{\lambda p s/2}\D s.
\]
By the H{\"o}lder inequality, 
\begin{eqnarray*}
\EE\left[\hat{L}(\hX_s,\bX_s)^{p/2}\|\hX_s\!-\!\bX_s\|^p\right]
\leq \left(\EE\left[\hat{L}(\hX_s,\bX_s)^{p}\right]\EE\left[\|\hX_s\!-\!\bX_s\|^{2p}\right]\right)^{1/2},
\end{eqnarray*}
and $\EE\left[\hat{L}(\hX_s,\bX_s)^{p}\right]$
can be bounded by a constant $C_p^1$ due to the stability property in Theorem \ref{thm:stability}.

For any $s\!\in\![0,T]$,
$\hX_s\!-\!\bX_s = f(\hX_{\sd}) (s\!-\!\sd) + g(\hX_{\sd}) (W_s \!-\! W_\sd)$,
and hence, by a combination of Jensen and H{\"o}lder inequalities, 
we get
\begin{eqnarray*}
\EE\left[ \|\hX_s\!-\!\bX_s\|^{2p}\right] 
&\leq& 2^{2p-1} \left( \EE\left[ \|f(\hX_{\sd})\|^{4p} \right]
\EE \left[ (s\!-\!\sd)^{4p} \right]\right)^{1/2}
\\ &+& 
2^{2p-1}  \left( \EE\left[ \|g(\hX_{\sd})\|^{4p} \right]
                 \EE\left[ \|W_s \!-\! W_\sd\|^{4p} \right] \right)^{1/2}.
\end{eqnarray*}
$\EE[ \|f(\hX_{\sd})\|^{4p}]$ and $\EE[ \|g(\hX_{\sd})\|^{4p}]$ are both finite,
due to stability and the polynomial bounds on the growth of $f(x)$ and $g(x)$.
Furthermore, we have 
$\EE[ (s\!-\!\sd)^{4p}]
\leq (\delta \hmax)^{4p} \leq \hmax^{4p}\delta^{2p} $,
and by standard results there is a constant $c_p$ such that
$\EE[ \|W_s \!-\! W_\sd\|^{4p}] 
= \EE[\ \EE[\|W_s \!-\! W_\sd\|^{4p}\, |\, {\cal F}_\sd] \ ]
\leq c_p (\delta \hmax)^{2p}$.
Hence, there exists a constant $C_p^2\!>\!0$ such that
$\EE[\, \|\hX_s\!-\!\bX_s\|^{2p}] \leq C^2_p\, \delta^p$,
and therefore equation (\ref{eq:halfway}) gives us
\[
\EE\left[ \e^{\lambda p t/2}\|e_t\|^p\right] \leq 2\left(\frac{8(p\!-\!2)}{p\lambda}\right)^{p/2-1}
\int_0^t \sqrt{C_p^1C_p^2}\ \delta^{p/2} \e^{\lambda p s/2} \,\D s,
\]
which provides the final result:
\[
\EE\left[ \|e_t\|^p\right]\leq \frac{4}{\lambda p}\left(\frac{8(p\!-\!2)}{p\lambda}\right)^{p/2-1}\sqrt{C_p^1C_p^2}\ \delta^{p/2}\triangleq C_p\, \delta^{p/2},\ \forall\ t\geq 0
\]
\end{proof}

\subsection{Theorem \ref{thm:convergence_order2}}

\begin{proof}
The proof is given for $p\!\geq\! 4$; the result for 
$0\!<\! p \!<\! 4$ follows from H{\"o}lder's inequality.

The error $e_t\triangleq \hX_t-X_t$ satisfies the SDE
$\D e_t = \left( f(\bX_t)\!-\!f(X_t) \right)\D t$
and hence by It\^{o}'s formula,
\begin{eqnarray}
\e^{2\lambda t}\|e_t\|^2
&=&\int_0^t 2\lambda\, \e^{2\lambda s}\|e_s\|^2\,\D s+
\int_0^t 2\, \e^{2\lambda s}\langle e_s, f(\hX_s)\!-\!f(X_s)\rangle \,\D s \nonumber \\
&&-\int_0^t 2\, \e^{2\lambda s}\langle e_s, f(\hX_s)\!-\!f(\bX_s) \rangle \,\D s
\nonumber \\&\leq &
- \int_0^t  2\, \e^{2\lambda s}\langle e_s, f(\hX_s)\!-\!f(\bX_s)\rangle \,\D s
\label{woch}
\end{eqnarray}
due to the one-sided Lipschitz condition (\ref{eq:onesided_Lipschitz2}), 
so therefore
\begin{eqnarray*}
\EE\left[\sup_{0\leq s\leq t}\e^{\lambda p s}\|e_s\|^p \right]
&\leq&\  
2^{p/2}\,\EE\left[\sup_{0\leq s\leq t}\left| 
\int_0^s \e^{2\lambda u}\langle e_u, f(\hX_u)\!-\!f(\bX_u) \rangle \,\D u\,
\right|^{p/2}\right]  
\end{eqnarray*}
Within a single timestep, 
$\hX_u\!-\!\bX_u = f(\bX_u)(u\!-\!\ud) + (W_u\!-\!W_\ud)$,
and therefore, following a similar approach to the proof of Theorem 4 in \cite{Part1}, Lemma \ref{lemma:useful} gives
\begin{eqnarray*}
\e^{2\lambda u}\langle e_u, f(\hX_u)\!-\!f(\bX_u) \rangle 
&=&\e^{2\lambda u} \langle e_u, \nabla f(\bX_u) (\hX_u\!-\!\bX_u) \rangle\ +\ \e^{2\lambda u}R_u \\[0.1in]
&=& \e^{2\lambda u}\langle e_u, (u\!-\!\ud) \nabla f(\bX_u) f(\bX_u) \rangle\ +\ \e^{2\lambda u}R_u \\
&&  +\ (\e^{2\lambda u}\!-\!\e^{2\lambda \ud})\langle e_u,  \nabla f(\bX_u) (W_u\!-\!W_\ud) \rangle \\
&& +\e^{2\lambda \ud}\langle ( e_u\!-\!e_\ud),  \nabla f(\bX_u) (W_u\!-\!W_\ud) \rangle\\
&&  +\ \e^{2\lambda \ud}\langle e_\ud, \nabla f(\bX_u) (W_u\!-\!W_\ud) \rangle
\end{eqnarray*}
where 
$|R_u|\leq \left( \gamma\, (\|\hX_u\|^q \!+\! \|\bX_u\|^q) + \mu\right)\, \|e_u\| \, \|\hX_u\!-\!\bX_u \|^2$,
and hence
\[
\EE\left[\sup_{0\leq s\leq t}\e^{\lambda p s}\|e_s\|^p \right]  \leq \frac{10^{p/2}}{5} (I_1 + I_2 + I_3 + I_4+I_5)
\]
where
\begin{eqnarray*}
I_1 & = & \EE\left[\sup_{0\leq s\leq t}\left|\int_0^s
\e^{2\lambda u} \langle e_u, (u\!-\!\ud)  \nabla f(\bX_u) f(\bX_u) \rangle
 \, \D u \,\right|^{p/2}\right], \\
I_2 & = & \EE\left[\sup_{0\leq s\leq t}\left|\int_0^s \e^{2\lambda u}R_u
 \ \D u \,\right|^{p/2}\right], \\
I_3 & = & \EE\left[\sup_{0\leq s\leq t}\left|\int_0^s(\e^{2\lambda u}\!-\!\e^{2\lambda \ud})\langle e_u, \nabla f(\bX_u) (W_u\!-\!W_\ud) \rangle
 \, \D u \,\right|^{p/2}\right],  \\
I_4 & = & \EE\left[\sup_{0\leq s\leq t}\left|\int_0^s \e^{2\lambda \ud}
\langle ( e_u\!-\!e_\ud), \nabla f(\bX_u) (W_u\!-\!W_\ud) \rangle
 \, \D u \,\right|^{p/2}\right], \\
I_5 & = & \EE\left[\sup_{0\leq s\leq t}\left|\int_0^s
\e^{2\lambda \ud}\langle e_\ud, \nabla f(\bX_u) (W_u\!-\!W_\ud) \rangle
 \, \D u \,\right|^{p/2}\right].
\end{eqnarray*}

We now bound $I_1, I_2, I_3, I_4, I_5$ in turn.  Noting that $u\!-\!\ud\leq \delta \hmax$, by Young inequality (\ref{Young inequality 2}) and Jensen inequality (\ref{Jensen exp}), we obtain
\begin{eqnarray*}
I_1 &\!\!\leq\!\!&
\frac{\e^{\lambda(p/2-1)t} }{\lambda^{p/2-1}}\EE\left[\int_0^t  \e^{\lambda pu/2}
\left( \|e_u\| \delta \hmax \| f(\bX_u)\|\ \|\nabla f(\bX_u)\|\right)^{p/2} \e^{\lambda u}\, \D u \right]
\\ 
&\!\!\leq\!\!&\frac{\e^{\lambda(p/2-1)t}}{\lambda^{p/2-1}}(\delta \hmax)^{p/2}\EE\left[\left(\sup_{0\leq s\leq t} \e^{\lambda ps/2}\|e_s\|^{p/2}\right)\!\!\! 
\right. \\ && \left. \hspace{1.75in}
\times \int_0^t \| f(\bX_u)\|^{p/2} \|\nabla f(\bX_u)\|^{p/2}\! \e^{\lambda u}\D u \right]   \\
&\!\!\leq\!\!&\xi \ \EE\left[\sup_{0\leq s\leq t} \e^{\lambda ps}\|e_s\|^{p}\right]
\\ && +\, \frac{\e^{\lambda(p-1)t}}{4\xi\, \lambda^{p-1}}\,(\delta \hmax)^{p}\int_0^t  \EE\left[ \| f(\bX_u)\|^{p} \|\nabla f(\bX_u)\|^{p}\right] \e^{\lambda u}\, \D u .
\end{eqnarray*}
The last integral is finite because of stability and the polynomial bounds on the growth
of both $f$ and $\nabla f$, and hence there is a constant $C^1_{p}$ such that
\[
I_1 \leq \xi\ \EE\left[\sup_{0\leq s\leq t} \e^{\lambda ps}\|e_s\|^{p}\right] + (C^1_{p}/\xi)\, \e^{\lambda p t}\, \delta^p.
\]
Similarly, using the Young inequality (\ref{Young inequality 2}), Jensen inequality (\ref{Jensen exp}) and the H\"{o}lder inequality, we obtain
\begin{eqnarray*}
I_2 &\!\leq\!& 
 \frac{\e^{\lambda(p/2-1)t}}{\lambda^{p/2-1}}\int_0^t \EE\!\left[ \e^{\lambda pu/2}
\|e_u\|^{p/2} L^{p/2}(\hX_u,\bX_u) 
\|\hX_u\!-\!\bX_u \|^p \right]\e^{\lambda u} \D u
\\&\!\leq\!& 
\xi \ \EE\left[\sup_{0\leq s\leq t} \e^{\lambda ps}\|e_s\|^{p}\right]
\\&& + \,\frac{\e^{\lambda(p-1)t}}{4 \xi\, \lambda^{p-1}}
\int_0^t \left( \EE\!\left[ L^{2p}(\hX_u,\bX_u)\right] \, 
 \EE\!\left[ \|\hX_u\!-\!\bX_u \|^{4p}\right]\right)^{\!1/2}\e^{\lambda u}\, \D u.
\end{eqnarray*}
where $L(\hX_u,\bX_u) \equiv  \gamma\, (\|\hX_u\|^q \!+\! \|\bX_u\|^q) \!+\!\mu$.
Hence, using stability and bounds on $\EE\left[ \|\hX_u\!-\!\bX_u \|^{4p} \right]$ 
from the proof of Theorem \ref{thm:convergence_order}, there is a constant $C^2_{p}$ such that
\[
I_2 \leq \xi\, \EE\left[\sup_{0\leq s\leq t} \e^{\lambda ps}\|e_s\|^{p}\right]
 + (C^2_{p}/\xi) \, \e^{\lambda p t}\, \delta^p.
\]
The fact that $\e^{2\lambda u}\!-\!\e^{2\lambda \ud}\leq 2\lambda\, \e^{2\lambda u}(u-\ud),$ and Jensen inequality (\ref{Jensen exp}) gives
\[
I_3 \leq
\frac{\e^{\lambda(p/2-1)t} }{\lambda^{p/2-1}}
\!\int_0^t\! \EE\left[ \e^{\lambda pu/2}
\left( \|e_u\|(2\lambda \delta \hmax)\|\nabla f(\bX_u)\|\, \|W_u-W_{\ud}\|\right)^{p/2}\right]\! \e^{\lambda u}\D u 
\]
and using the approach to estimating $I_1,$ there is similarly
a constant $C_p^3$ such that
\[
I_3\leq \xi \, \EE\left[\sup_{0\leq s\leq t} \e^{\lambda ps}\|e_s\|^{p}\right] 
+ (C^3_{p}/\xi) \, \e^{\lambda p t}\, \delta^p.
\]  
For the next term, $I_4$, we start by bounding $\|e_s\!-\!e_\sd\|$.  Since
\[
e_s\!-\!e_\sd = \int_\sd^s \left(f(\bX_u) - f(X_u)\right) \ \D u,
\]
by Jensen's inequality and Assumption \ref{assp:enhanced_Lipschitz} it follows that
\begin{eqnarray*}
\| e_s\!-\!e_\sd \|^p 
   &\leq& (\delta \hmax)^{p-1} \int_\sd^s \| f(\bX_u) \!-\! f(X_u) \|^p \, \D u
\\ &\leq& (2\delta \hmax)^{p-1}
\int_\sd^s L^p(\bX_u,X_u) \left( \|e_u\|^p + \| \hX_u\!-\!\bX_u\|^p \right) \, \D u,
\end{eqnarray*}
where $L(\bX_u,X_u)\equiv \gamma(\|\bX_u\|^q\!+\!\|X_u\|^q) + \mu$.
We again have an $O(\delta^{p/2})$ bound for $\EE[ \|\hX_s\!-\!\bX_s\|^p]$,
while Theorem \ref{thm:convergence_order} proves that there is a 
constant $c_{p}$ such that
\[
\EE[ \| e_s \|^p] \leq c_{p} \, \delta^{p/2}.
\]
Combining these, and using the H{\"o}lder inequality and the finite 
bound for $\EE[L^p(\bX_u,X_u)]$ for all $p\!\geq\! 2$, due to the 
usual stability results, we find that there is a different constant 
$c_{p}$ such that
\[
\EE[ \| e_s\!-\!e_\sd \|^p ] \leq c_{p} \, \delta^{3p/2}.
\]
Now, by Jensen inequality (\ref{Jensen exp}), we obtain
\begin{eqnarray*}
I_4 \leq \frac{\e^{\lambda(p-2)t}}{(2\lambda)^{p/2-1}} \int_0^t \EE\left[ 
\| e_s\!-\!e_\sd\|^{p/2}\|\nabla f(\bX_s)\|^{p/2} \|W_s\!-\!W_\sd\|^{p/2} \right]\e^{2\lambda s}\D s,
\end{eqnarray*}
so using the H{\"o}lder inequality and the usual stability bounds, 
we conclude that there is a constant $C^4_{p}$ such that
\[
I_4 \leq  C^4_{p}\, \e^{\lambda p  t}\, \delta^p.
\]
Lastly, 
considering that 
\[
\D \left((t\!-\!t_{n+1})(W_t\!-\!W_{t_n})\right)=(W_t\!-\!W_{t_n})\,\D t+(t\!-\!t_{n+1})\,\D W_t
\] 
in the time interval $[t_n,t_{n+1}]$ conditioned on $\mathcal{F}_{t_n},$ we have
\[
\int_{t_n}^{t_{n+1}}(W_u\!-\!W_{t_n})\,\D u 
= - \int_{t_n}^{t_{n+1}}(u\!-\!t_{n+1})\,\D W_u,
\]
and therefore we can split $I_5$ into two parts:
\begin{eqnarray*}
I_5&\leq& 2^{p/2-1}\left\{\EE\left[\sup_{0\leq s\leq t}\left| \sum_{k=0}^{n_s-1} \int_{t_k}^{t_{k+1}} \e^{2\lambda\ud}(u\!-\!t_{k+1})\langle e_{t_k},\nabla f(\bX_{t_k})\D W_u\rangle \right|^{p/2}\right] \right.\\
&& \left. ~~~~~~~~+\EE\left[ \sup_{0\leq s\leq t}\left| \int_{\sd}^s
\e^{2\lambda \ud}\langle e_\ud, \nabla f(\bX_\ud) (W_u\!-\!W_\ud) \rangle
\,\D u\right|^{p/2} \right]\right\}\\
 &:= & 2^{p/2 - 1}\, (I_{51} + I_{52}).
\end{eqnarray*}
By the Burkholder-Davis-Gundy inequality, Young inequality (\ref{Young inequality 2}) and Jensen inequality (\ref{Jensen exp}), there exists a constant $c_p^1$ such that
\begin{eqnarray*}
I_{51} &\leq& \EE\left[\sup_{0\leq s\leq t}\left|  \int_{0}^{s} \e^{2\lambda\ud}(u-\ud-h(\bX_{\ud}))\langle e_{\ud},\nabla f(\bX_{\ud})\D W_u\rangle \right|^{p/2}\right]\\
& \leq &\EE\left[c_p^1 \left( \int_{0}^{t} \e^{4\lambda\ud}(\delta \hmax)^2\|e_{\ud}\|^2\|\nabla f(\bX_{\ud})\|^2\D u \right)^{p/4}\right]\\
&\leq & \EE\left[c_p^1 \sup_{0\leq s\leq t} \e^{\lambda p s/2}\|e_s\|^{p/2} \left( \int_{0}^{t} (\delta \hmax)^2\|\nabla f(\bX_{\ud})\|^2 \e^{2\lambda u}\D u \right)^{p/4}\right]\\
&\leq & \xi \,\EE\left[\sup_{0\leq s\leq t} \e^{\lambda p s}\|e_s\|^{p} \right]\\
&&+\frac{1}{4 \xi}\,\EE\left[\left(c_p^1\right)^2 \frac{\e^{\lambda(p-2)t}}{(2\lambda)^{p/2-1}} \int_0^t(\delta \hmax)^p \|\nabla f(\bX_{\ud})\|^p \e^{2\lambda u} \D u  \right]
\end{eqnarray*}

Hence, there exists a constant $C^{51}_{p}$ such that
\[
I_{51} \leq 
\xi \ \EE\left[\sup_{0\leq s\leq t} \e^{\lambda ps}\|e_s\|^{p}\right]
+ (C^{51}_{p}/\xi) \ \e^{\lambda p t}\, \delta^p.
\]
Finally, for $I_{52}$, the Young inequality (\ref{Young inequality 2}) 
and Jensen inequality (\ref{Jensen exp}) implies
\begin{eqnarray*}
I_{52} &\!\!=\!\!&
\EE\left[\sup_{0\leq s\leq t}\left| \int_\sd^s
\langle e_\ud, \nabla f(\bX_u) (W_u\!-\!W_\ud)\rangle \e^{2\lambda \ud}\,\D u\right|^{p/2}\right]
\\&\!\!\leq\!\!&
(\delta \hmax)^{p/2-1} \EE\left[\sup_{0\leq s\leq t} \int_\sd^s
\| e_\ud\|^{p/2} \|\nabla f(\bX_u)\|^{p/2} \|W_u\!-\!W_\ud\|^{p/2} \e^{\lambda p \ud}\,\D u\right]
\\&\!\!\leq\!\!&
(\delta \hmax)^{p/2-1} \EE\left[ \left(\sup_{0\leq s\leq t} \e^{\lambda p s/2} \|e_s\|^{p/2}\right)
\right. \\ && \left. \hspace{1.25in}
\times \int_0^t
\|\nabla f(\bX_u)\|^{p/2} \|W_u\!-\!W_\ud\|^{p/2} \e^{\lambda p u/2}\,\D u\ \right]
\\ &\!\!\leq\!\!&
\xi \ \EE\left[\sup_{0\leq s\leq t} \e^{\lambda ps}\|e_s\|^{p}\right]
\\ && +\, \frac{\e^{\lambda p t/2}}{2 \xi\, \lambda p} (\delta \hmax)^{3p/2-2} \!\!\!
\int_0^t\EE\left[ \|\nabla f(\bX_u)\|^{p}\right]\e^{\lambda pu/2} \D u.
\end{eqnarray*}
Thus, we find that there exists a constant $C^{52}_{p}>0$ such that
\begin{equation}
I_{52} \leq \xi \ \EE\left[\sup_{0\leq s\leq t} \e^{\lambda ps}\|e_s\|^{p}\right]
+ (C^{52}_{p}/\xi) \ \e^{\lambda p t}\, \delta^p.
\end{equation}
Combining the five bounds, and choosing $\xi$ to be sufficiently small,
we conclude that there is a constant $C_p$ such that
\[
\EE\left[\sup_{0\leq s\leq t}\e^{\lambda p s}\|e_s\|^p \right] 
\leq C_{p}\, \e^{\lambda p t}\, \delta^p.
\]
and therefore
\[
\EE\left[\|e_t\|^p \right] \leq C_{p}\, \delta^p,\ \forall\ t\geq 0.
\]
\end{proof}

\section{Conclusions and future work}
In this paper, we have developed the adaptive Euler-Maruyama method for a class of ergodic SDEs with non-globally Lipschitz drifts and shown the moments and strong error of the numerical solutions are uniformly bounded in time $T.$ Moreover, we extend this adaptive scheme to Multi-level Monte Carlo for expectations with respect to the invariant measure by using different time intervals on different levels, and constructing an efficient coupling of the fine path and coarse paths with different simulation times. Numerical experiments support the theoretical analysis.

One direction for extension of the theory is to address SDEs which don't satisfy the contractive property. For example, the SDEs arising from the double-well potential energy, in which case, the fine path and the coarse path may converge to different wells resulting in a poor coupling. Another possibility is to apply adaptive methods to SDEs with a discontinuous drift.

\newpage

\bibliographystyle{plain}
\bibliography{citation}

\end{document}